\documentclass[12pt,reqno]{amsart}

\usepackage[utf8]{inputenc}

\usepackage{bbold}

\usepackage{amssymb,latexsym}
\usepackage{amsfonts}
\usepackage{amsmath}
\usepackage{enumerate}

\usepackage{mathtools}

\usepackage{enumerate}

\usepackage{booktabs}

\usepackage[usenames, dvipsnames]{xcolor}
\usepackage[colorlinks,linkcolor=blue,anchorcolor=blue,citecolor=red]{hyperref}

\usepackage{amsrefs}

\usepackage{soul}

\usepackage{bm}
\usepackage{enumerate}
\usepackage{a4wide}
\usepackage{hyperref}

\usepackage{xcolor}

\usepackage{dsfont}

\usepackage{booktabs}

\usepackage{graphicx}

\usepackage{caption}
\usepackage{subcaption}

\newtheorem{theorem}{Theorem}

\newtheorem{corollary}[theorem]{Corollary}
\newtheorem{lemma}[theorem]{Lemma}
\newtheorem{proposition}[theorem]{Proposition}
\newtheorem{remark}[theorem]{Remark}

\newtheorem{problem}[theorem]{Problem}

\newtheorem{definition}[theorem]{Definition}

\newtheorem{example}[theorem]{Example}
\numberwithin{equation}{section}

\def\R{\mathbb{R}}
\def\C{\mathbb{C}}

\def\N{\mathbb{N}}
\def\Z{\mathbb{Z}}
\def\Q{\mathbb{Q}}

\def\mc{\mathcal{C}}

\def\al{\alpha}
\def\be{\beta}

\def\la{\lambda}

\def\aa{\mathcal{A}}
\def\dd{\mathcal{D}}
\def\ll{\mathcal{L}}
\def\mm{\mathcal{M}}
\def\pp{\mathcal{P}}
\def\rr{\mathcal{R}}
\def\sss{\mathcal{S}}
\def\tt{\mathcal{T}}

\def\vv{\mathcal{V}}

\def\id{\mathbb{1}}

\def\bb{\mathbf{b}}
\def\bc{\mathbf{c}}
\def\bd{\mathbf{d}}
\def\boe{\mathbf{e}}

\def\bn{\mathbf{n}}

\def\bu{\mathbf{u}}
\def\bv{\mathbf{v}}
\def\bw{\mathbf{w}}
\def\bx{\mathbf{x}}
\def\by{\mathbf{y}}
\def\bz{\mathbf{z}}
\def\bo{\mathbf{0}}
\def\beps{\bm{\varepsilon}}

\def\bdd{\dd}
\def\dsum{\oplus}

\newcommand{\abs}[1]{\left|#1\right|}
\newcommand{\norm}[1]{\left\Vert#1\right\Vert}
\newcommand{\cone}[1]{\text{Cone}\left(#1\right)}
\newcommand{\conv}[1]{\text{Conv}\left(#1\right)}

\newcommand{\rp}[2]{\text{Rep}_{#1}\left(#2\right)}
\newcommand{\twist}[3]{{#1}\oplus_{\scriptscriptstyle #3}{#2}}

\newcommand{\vectwo}[2]{\begin{pmatrix*}[r] #1 \\ #2 \end{pmatrix*}}
\newcommand{\vecfour}[4]{\begin{pmatrix*}[r] #1 \\ #2 \\ #3 \\ #4\end{pmatrix*}}

\newcommand{\vertiii}[1]{{\left\vert\kern-0.25ex\left\vert\kern-0.25ex\left\vert #1 
    \right\vert\kern-0.25ex\right\vert\kern-0.25ex\right\vert}}

\begin{document}

\title[Rational matrix digit systems]{Rational matrix digit systems}

\author{Jonas~Jankauskas}
\author{J\"org~M.~Thuswaldner}

\address[J.J.]{Faculty of Mathematics and Informatics, Institute of Mathematics, Vilnius University, Naugarduko g. 24, LT 03225, Vilnius, Lithuania}
\address[J.J and J.M.T.]{Mathematik und Statistik, Montanuniversit\"at Leoben, Franz Josef Stra\ss{}e 18, A-8700 Leoben, Austria}
\email{jonas.jankauskas@gmail.com}
\email{joerg.thuswaldner@unileoben.ac.at}

\thanks{The post-doctoral position of the first author at Montanuniversit\"at Leoben in 2018-2021 was supported by the Austrian Science Fund (FWF) project M2259 Digit Systems, Spectra and Rational Tiles under the Lise Meitner Program. The second author is supported by project I 3466 \emph{GENDIO} funded by the Austrian Science Fund (FWF)}


\begin{abstract}
Let $A$ be a $d \times d$ matrix with rational entries which has no eigenvalue $\la \in \C$ of absolute value $\abs{\la} < 1$ and let $\Z^d[A]$ be the smallest nontrivial $A$-invariant $\Z$-module. We lay down a theoretical framework for the construction of {\em digit systems} $(A, \bdd)$, where $\bdd\subset \Z^d[A]$ finite, that admit finite expansions of the form 
\[
\bx= \bd_0 + A \bd_1 + \cdots + A^{\ell-1}\bd_{\ell-1} \qquad(\ell\in \N,\;\bd_0,\ldots,\bd_{\ell-1} \in \bdd)
\]
for every element $\bx\in \Z^d[A]$. 
We put special emphasis on the explicit  computation of {\it small digit sets} $\bdd$ that admit this property for a given matrix $A$, using techniques from matrix theory, convex geometry, and the Smith Normal Form. Moreover, we provide a new proof of general results on this {\em finiteness property} and recover analogous finiteness results for digit systems in number fields a unified way.
\end{abstract}

\maketitle

\bigskip

\section{Introduction and main results}\label{sec:intro}
Let $A \in \Q^{d \times d}$ be an invertible matrix with rational entries. It is easy to see that the smallest nontrivial $A$-invariant $\Z$-module containing $\Z^d$ is given by the set $\Z^d[A]$ of vectors $\bx \in \mathbb{Q}^d$ that can be written as
\begin{equation}\label{eq:zrep}
\bx= \bx_0 + A \bx_1 + \cdots + A^{k-1}\bx_{k-1},
\end{equation}
where $k\in \N$ and $\bx_0,\ldots,\bx_{k-1} \in \Z^d$, {\it i.e.}, 
\[
\Z^d[A] := \bigcup_{k=1}^{\infty} \left(\Z^d + A\Z^d + \dots + A^{k-1}\Z^d\right).
\]
Thus $\Z^d[A]$ is the set of all vectors that can be expanded in terms of positive powers of $A$ with integer vectors as coefficients. Motivated by the existing vast theory and many practical applications of number systems, it is natural to ask the following question. Is there a finite set $\bdd \subset \Z^d[A]$, such that every vector $\bx \in \Z^d[A]$ has an expansion of the form 
\begin{equation}\label{eq:zrep2}
\bx= \bd_0 + A \bd_1 + \cdots + A^{\ell-1}\bd_{\ell-1},
\end{equation}
where $\ell\in \N$ and $\bd_0,\ldots,\bd_{\ell-1} \in \bdd$. To put this in another way: can one find a finite digit set $\bdd\subset \Z^d[A]$ such that 
\begin{equation}\label{eq:ZdD}
\Z^d[A]= \bdd[A], 
\end{equation}
where
\begin{equation}\label{eq:defDA}
\bdd[A] := \bigcup_{\ell=1}^{\infty} \left(\bdd + A\bdd + \dots + A^{\ell-1}\bdd\right).
\end{equation}

A lot of work related to this problem has been done in the context of number systems in number fields ({\it cf.}~\cites{KovPet1,KovPet2,Brunotte:01,EvGyPe19}),  polynomial rings (see~\cites{AkPe:02,AkRa,Petho:91, SSTW, Woe1} and the survey~\cites{BaratBertheLiardetThuswaldner06}), lattices (see {\it e.g.}\ \cites{KoLa,Vin2}), and rational bases ({\it cf.}~\cite{AkFrSa}). Moreover, it is intimately related to the study of self-affine tiles (see~\cites{KLR:04,LagWan97,StTh15}) and the so-called {\em height reducing property} ({\it cf.}~\cites{AkDrJa,AkThZa,AkiZai, JanThu}). As we will see below, the problem of the existence of such a set $\bdd$ can be restated in terms of a so-called \emph{finiteness} property of general \emph{digit systems}. It is the aim of the present paper to construct sets $\bdd$ that satisfy \eqref{eq:ZdD}. Here we strive for very general theory on the one side while paying attention to algorithmic issues on the other side.

 
The pair $(A, \bdd)$ is called \emph{a matrix digit system} or simply \emph{a digit system}. The matrix $A$ is then called \emph{the base} of this digit system and the elements of $\bdd$ are called its \emph{digits}. We say that $(A, \bdd)$ has \emph{the finiteness property} in $\Z^d[A]$ (or Property (F) for short), if every vector $\bx \in \Z^d[A]$ has a finite expansion of the form \eqref{eq:zrep2}, {\it i.e.}, if \eqref{eq:ZdD} holds. It is easily seen that the requirement that $\bdd$ contains a complete system of residue class representatives of  $\Z^d[A]/A\Z^d[A]$ is a necessary, but in general not sufficient, condition for $(A, \bdd)$ to have the finiteness property. One also says that $(A, \bdd)$ possesses \emph{the uniqueness property}, or Property (U), if two different finite radix expressions in \eqref{eq:defDA} always yield different vectors of $\Z^d[A]$. For the number systems that already possess Property (F), a necessary and sufficient condition to have Property (U) is $|\bdd|= |\Z^d[A]/A\Z^d[A]|$.

Digit systems that possess both properties (F) and (U) are called \emph{standard}. Frequently, one additional requirement is imposed on standard digit systems: the zero vector $\bo_d$ must belong to $\bdd$ (which enables natural positional arithmetics). Standard digit systems in $\Z^d$ in integer matrix bases have been studied extensively in the literature.

Clearly, matrix digit systems are modeled after the usual number systems in $\Z$ (see \cite{Grunwald:1885} for the negative case). Concerning rational matrices, in dimension $d=1$ one recovers the base $A=\tfrac pq \in \Q$, with $p \in \Z$, $q \in \N$, $\gcd(p, q)=1$. In this case, $\Z[\tfrac pq]=\Z[\tfrac1q]$ is the set of all rational fractions with denominators that divide some power of $q$, and the submodule $\tfrac pq\Z[\tfrac pq]=p\Z[\tfrac 1q]$ represents fractions with numerators divisible by $p$; so that $\bdd\{0, 1, \dots, p-1\}$ is a complete set of  coset representatives of $\Z[\tfrac pq]$ modulo $\tfrac pq\Z^d[\tfrac pq]$. The pair $(\tfrac pq,\bdd)$ then is a rational matrix  number system in $\Z[\tfrac 1q]$. Such systems have been studied in \cite{AkFrSa}.

If $A \in \Z^{d \times d}$ is an integer matrix, the module $\Z^d[A]$ reduces to the lattice $\Z^d$. Digit systems in lattices with expanding base matrices were introduced by Vince \cites{Vin1, Vin2, Vin3} and studied extensively in connection with  fractal tilings. Indeed, a tiling theory for the standard digit systems with expanding integral base matrices was worked out by Gr{\"o}chenig and Haas \cite{GroHaa94}, and Lagarias and Wang \cites{LagWan96a, LagWan96b, LagWan97}. A corresponding theory for tilings produced by rational algebraic number base systems was developed by Steiner and Thuswaldner~\cite{StTh15}.

When trying to study the finiteness property of digit systems $(A, \bdd)$ for large classes of base matrices $A$ and digits sets $\bdd$, new challenges of geometrical and arithmetical nature arise. First, if $A$ is allowed to have eigenvalues $\la \in \C$ of absolute value $\abs{\la}=1$, then $A^{-1}$ is no longer a contracting linear mapping. If such $\la$ are not roots of unity, then there exist points $\bx \in \R^n$ with infinite, non-periodic orbits $(A^{-n}\bx)_{n=1}^{\infty}$. Even worse, if $A$ has a Jordan block of order $ \geq 2$ corresponding to the eigenvalue $\la$ on the complex unit circle in its Jordan decomposition, then $A^{-n}\bx$ can diverge to $\infty$, as it is easily seen from the simple $2 \times 2$ example $A=\begin{pmatrix}1 & 1\\0 & 1\end{pmatrix}$ by taking $\bx=\begin{pmatrix}
0\\1    
\end{pmatrix}$. These situations require careful control of vectors in certain directions when multiplication by $A^{-1}$ is performed. Aside from the issues of geometric nature, arithmetics in $\Z^d[A]$ is also more difficult when $A$ has rational entries: even the computation of the residue class group for the expansions in base $A$ is no longer trivial and causes significant problems.

As a result of the aforementioned challenges, properties (F) and (U) no longer go well together: we typically need \emph{more} than $[\Z^d[A]:A\Z^d[A]]$ digits to address all possible issues. In this situation, one is forced to make a choice, which property, (F) or (U), must be sacrificed to reinforce another. We make the choice in the favor of Property (F): it is often important to have multiple digital expressions for every $\bx \in \Z^d[A]$, rather than leave some $\bx$ with no such expression whatsoever while trying to preserve (U). 

The aim of the present paper is to ``build up'' digit sets with Property (F) in an explicit and algorithmic way. We reveal the  dynamical systems underlying the digit systems $(A,\bdd)$ and interpret the finiteness property in terms of their attractors. The properties of these dynamical systems are investigated in Section~\ref{secArithm}. After that, we decompose a given rational matrix $A$ into ``generalized Jordan blocks'', called \emph{the hypercompanion matrices},  build digits systems for each such block separately, and then patch them together (see Section~\ref{sec:twist}). In these building blocks, {\em generalized rotation matrices} play a crucial role in order to treat the eigenvalues of modulus $1$. These matrices and their digit sets are thoroughly studied in Section~\ref{secPerGenRot}. The decomposition of $A$ has the advantage that we get much better hold on the size and structure of the digit sets than in previous papers (like \cites{AkThZa,JanThu}).  This can be used to give new proofs of the main results of \cites{AkThZa,JanThu} which can easily be turned in algorithms in order to construct convenient digit sets for any given matrix $A$ (see Section~\ref{sec:general}). 
The algorithmic nature of our approach is  exploited in the discussion of various examples. In particular, in Section~\ref{sec:theorySNF} we use the Smith Normal Form in order to get results on the size of certain residue class rings that are relevant for the construction of digit sets with finiteness property.
In Section~\ref{sec:num} these results are combined with our general theory
in order to construct examples of small digit sets for matrices with particularly interesting properties.

Summing up, in the present paper we lay down a comprehensive (and much simplified) theory for the finiteness property in digit systems with rational matrix bases. Our construction of the digit sets is based on the theory of matrices, convex geometry and Smith Normal Form representations of lattice bases; these proofs are independent from the previous constructions of \cites{AkThZa,JanThu}, that were based on digit systems in number fields. In fact, we show that our finiteness result on matrix-base digit systems implies the according result on digit systems in number fields, thereby demonstrating that these two finiteness results are of the same strength. 

\section{Arithmetics of matrix digit systems}\label{secArithm}

In this section we discuss discrete dynamical systems associated that represent 'the remainder division' in the module $\Z^d[A]$. If these dynamical systems have finite attractors, this enables us to construct a digit system $(A,\bdd)$ with Property (F). We also consider certain restrictions of these dynamical mappings to the integer lattice $\Z^d$ and their subsequent extensions to the whole module $\Z^d[A]$ which appear to be very useful in the practical computations. 

\subsection{Classical remainder division}

Let $A \in \Q^{d \times d}$ be given and let $\bdd \subset \Z^d[A]$ be a digit set.
We take a closer look at the arithmetics of $\Z^d[A]$, in particular, the division with remainder. Associated with the digit system $(A, \bdd)$ is the procedure of division with remainder in $\Z^d[A]$. As usual, a mapping $\bd: \Z^d[A] \to \bdd$ is called \emph{a digit function} if $\bd(\bx) \equiv \bx \pmod{A\Z^d[A]}$, {\it i.e.} if $\bd(\bx)-\bx \in A\Z^d[A]$. Thus, using $\bd$ we can define the dynamical system
\begin{equation}\label{eq:defRemDiv}
\Phi: \Z^d[A] \to \Z^d[A], \qquad \Phi(\bx) := A^{-1}(\bx - \bd(\bx)).
\end{equation}

As we allow $A \in \Q^{d \times d}$ to have non-integral coefficients, our first step is to show that division with remainder works essentially in the same way, as in the classical case. In this context, the following definition is of importance. 

\begin{definition}[Attractor of $\Phi$]\label{def:Attr}
A set $\aa \subset{\Z^d[A]}$ with the properties that:
\begin{enumerate}[i)]
\item $\Phi(\aa) \subset \aa$,
\item for every $\bx \in \Z^d[A]$, there exists $n=n(\bx) \in \N$, such that $\Phi^n(\bx) \in \aa$
\item no proper subset of $\aa$ satisfies properties $(i-ii)$
\end{enumerate}
 is called \emph{an attractor} of $\Phi$ and will be denoted $\aa_{\Phi}$.
\end{definition}

It should be noted that multiple variants of Definition \ref{def:Attr} are used in the literature (cf. \cite{BrinStuck:2015}*{Section 1.13}). In general, the attractor might not exist, or be infinite. However, if the attractor $\aa_{\Phi}$ exits, Definition \ref{def:Attr} implies it is unique. If the attractor $\aa_{\Phi}$ of $\Phi$ is finite, then one can construct a digit system $(A,\dd)$ that satisfies Property (F).

\begin{proposition}\label{prop:finAttr}
Suppose that the attractor $\aa_{\Phi}$ of the quotient function $\Phi$ is a finite set. Then, for every $\bx \in \Z^d[A]$, the orbit $(\Phi^n(\bx))_{n =1}^{\infty}$ is eventually periodic, with only a finite number of possible periods. In particular, the digit system $(A, \bdd \cup \aa_{\Phi})$ has the finiteness property. 
\end{proposition}
\begin{proof}[Proof of Proposition \ref{prop:finAttr}] By Definition~\ref{def:Attr}, for each $\bx \in \Z^d[A]$, there is $n_0(\bx) \in \N$, such that $
\bx = \beps_0 + A\beps_1 + \dots + \beps_{n-1}A^{n-1} + A^n\Phi^n(\bx),
$
with $\beps_1,\ldots,\beps_{n-1} \in \bdd$ and $\Phi^n(\bx) \in \aa_{\Phi}$ holds for each $n\ge n_0(\bx)$. Thus $\Z^d[A]=(\bdd \cup \aa_\Phi)[A]$ and, hence, $(A, \bdd \cup \aa_{\Phi})$ has the finiteness property.
As the value of $\Phi^{n+1}(\bx)$ depends only on $\Phi^n(\bx)$, the statement on the ultimate periodicity and periods follows from the finiteness of $\aa_{\Phi}$.
\end{proof}

Suppose that $(A,\bdd)$ is a digit system.
Then we can use the dynamical system $\Phi$ in order to generate expansions of $x\in \Z^d[A]$. Indeed, by the definition of $\Phi$ there are uniquely defined elements $\bd_0,\bd_1,\ldots \in \bdd$ satisfying
\[
\bx = \beps_0 + A\beps_1 + \dots + A^{n-1}\beps_{n-1} + A^n\Phi^n(\bx) \qquad(n\in \N).
\]
If the attractor of $\Phi$ is $\{\bo_d\}$ this implies that for each $\bx\in \Z^d[A]$ there is $n\in \N$ such that 
\begin{equation}\label{eq:xgenphi}
\bx = \beps_0 + A\beps_1 + \dots + A^{n-1}\beps_{n-1}.
\end{equation}
In this case we say that \eqref{eq:xgenphi} is \emph{a finite expansion of $\bx$ generated by $\Phi$}.

For number systems in $\Z^d$ with integer matrices as base it is well-known that the attractor of a number system $(A,\bdd)$ with an expanding matrix $A\in\Z^{d\times d}$ and a digit set $\bdd \subset \Z^d$ is a finite set (see {\it e.g.}~\cite{Vin2}*{Section~5}). Our next goal is to show that the same result holds in our more general context.
Let $A \in \Q^{d \times d}$ be an invertible rational matrix. As the entries of the vectors in $\Z^d[A]$ may have arbitrarily large denominators, $\Z^d[A]$, in general, does not possess any finite $\Z$-basis. This entails that our proofs become a bit more complicated than in the integer case.
In particular we will frequently need the following sequence of rational lattices.
Starting with the lattice $\Z^d$ and repeatedly multiplying its vectors by $A$, $k=1, 2, \dots$ times, one defines
\[
\Z_k^d[A] := \Z^d + A\Z^d + \dots + A^{k-1}\Z^d.
\]
Since $\Z_k^d[A] \subset \Z_{k+1}^d[A]$, these rational lattices form a nested chain. 

We need two preparatory lemmas.

\begin{lemma}\label{lem:Res} If $A \in \Q^{d \times d}$ is invertible, then $\Z^d[A]/A\Z^d[A]$ is a finite Abelian group. Moreover, there exists $k := k_A \in \N$, such that 
\begin{equation}\label{eq:ZZAcap}
\Z^d \cap A\Z^d[A] = \Z^d \cap A\Z_k^d[A], \quad\hbox{and}\quad \Z^d[A]\big/A\Z^d[A] = \Z^d \big/ \left(\Z^d \cap A\Z^d_k[A]\right). 
\end{equation}
In particular, representatives for $\Z^d[A]/A\Z^d[A]$ can be chosen from $\Z^d$.
\end{lemma}
\begin{proof}[Proof of Lemma \ref{lem:Res}]
As $\Z^d[A]=\Z^d+A\Z^d[A]$, the Second Isomorphism Theorem for modules yields
\begin{equation}\label{eq:seconIsom}
\Z^d[A] \big/ A\Z^d[A] = \left(\Z^d+A\Z^d[A]\right)\big/A\Z^d[A] \simeq \Z^d\big/\left(\Z^d \cap A\Z^d[A]\right).
\end{equation}
Since $A$ is invertible, $\Z^d \cap A\Z^d[A]$ is a $d$-dimensional additive subgroup of $\Z^d$. Hence, it is a lattice and has a finite index in $\Z^d$. Therefore, $\Z^d[A]/A\Z^d[A]$ is a finite Abelian group.
Now, consider the nested chain of lattices
\begin{equation*}
\Z^d \cap A\Z^d_1[A] \subset \dots \subset \Z^d \cap A\Z^d_j[A] \subset  \Z^d \cap A\Z^d_{j+1}[A] \subset \dots \subset \Z^d \cap A\Z^d[A] \subset \Z^d.
\end{equation*}
Since the index $\Z^d \cap A\Z_j^d[A]$ in $\Z^d$ cannot decrease indefinitely, the chain must eventually stabilize, hence, there is a constant $k=k_A$ such that $\Z^d \cap A\Z_{j+1}^d[A] = \Z^d \cap A\Z_j^d[A]$ holds for $j \geq k$. This yields the first identity in \eqref{eq:ZZAcap}. Thus in order to establish the second identity it is sufficient to prove equality (instead of isomorphy) in \eqref{eq:seconIsom}. From $\Z^d[A] = \Z^d + A\Z^d[A]$ and $\Z^d = \rr + \Z^d \cap A\Z^d[A]$, where $\rr$ is a complete set of coset  representatives of $\Z^d/(\Z^d \cap A\Z^d[A])$, it follows that $\Z^d[A] = \rr + A\Z^d[A]$. Thus, $\rr$ contains all coset representatives of $\Z^d[A]/A\Z^d[A]$. $|\rr|=|\Z^d[A]/A\Z^d[A]|$ by \eqref{eq:seconIsom}, these representatives also must be different modulo $A\Z^d[A]$. Therefore, the isomorphism 
``$\simeq$'' symbol in \eqref{eq:seconIsom} can indeed be replaced with ``$=$''.
\end{proof}

\begin{lemma}\label{lem:RemDiv}
There exists $l:=l(A, \bdd) \in \N$, such that, for every $\bx \in \Z^d[A]$ and sufficiently large $n \geq n_{\bx}$, $\Phi^{n}(\bx) \in \Z_l^d[A]$.
\end{lemma}
\begin{proof}[Proof of Lemma \ref{lem:RemDiv}]
According to Lemma \ref{lem:Res}, there exists $k=k_A$, such that $\Z^d \cap A\Z^d[A] = \Z^d \cap A\Z_k^d[A]$. Choose the smallest integer $l \geq k_A$, such that $\bdd \subset \Z_l^d[A]$. Assume that $\bx = \bx_0 + A\bx_1 + \dots + A^{m-1}\bx_{m-1} \in \Z^d_m[A]$, with $\bx_0,\ldots,\bx_{m-1}\in \Z^d$. Because $\bd(\bx)\in \bdd$ there are 
$\beps_0,\ldots,\beps_{l-1} \in \Z^d$ such that
$\bd(\bx) = \beps_0 + A\beps_1 + \dots+ A^{l-1}\beps_{l-1}$.
Therefore we have $\bx - \bd(\bx) \in (\bx_0 - \beps_0 + A\Z^d_{i}[A]) \cap A\Z^d[A]$, with $i=\max\{l-1,m-1\}$ from which we conclude that $\bx_0 - \beps_0 \in \Z^d \cap A\Z^d[A] = \Z^d \cap A\Z_k^d[A] \subset A\Z_l^d[A]$. This implies that $\Phi(\bx) \in \Z_j^d[A]$, where $j=\max\{l, m-1\}$. By iteration, eventually $\Phi^n(\bx) \in \Z_l^d[A]$.
\end{proof}

We are now in a position to prove that expanding matrices lead to finite attractors.

\begin{proposition}\label{prop:finExp} If $A \in \Q^{d \times d}$ is expanding and $\bdd$ is a digit set, then the quotient function $\Phi$ of $(A,\bdd)$ has a finite attractor $\aa_{\Phi} \subset \Z^d[A]$. Consequently, the digit system $(A, \bdd \cup \aa_{\Phi})$ has the finiteness property in $\Z^d[A]$.
\end{proposition}
\begin{proof}[Proof of Proposition \ref{prop:finExp}] By Lemma \ref{lem:RemDiv}, there exists $l \in \N$, such that $\Phi^n(\bx) \in \Z_l^d[A]$ for every $n \geq n_1(\bx)$. From $\bx = \beps_0 + A\beps_1 + \dots + \beps_{n-1}A^{n-1} + A^n\Phi^n(\bx)$, $\beps_0,\ldots,\beps_{n-1} \in \bdd$, one deduces $\Phi^n(\bx)= A^{-n}\bx -A^{-1}\beps_{n-1}-A^{-2}\beps_{n-2}-\dots-A^{-n}\beps_0$. Since $A$ is expanding, the operator norm series $\sum_{n=0}^{\infty}\norm{A^{-n}}_{\infty} < C_A < \infty.$ It means $\norm{\Phi^n(\bx)} \leq C_A\max\{\norm{\beps}\,:\, \beps\in \bdd\}$. As $\Z_l^d[A]$ is a lattice, $\Z_l^d[A] \cap B(\bo_d, r)$ must be finite. Therefore, $\Phi$ has a finite attractor $\aa_\Phi \subset \Z_l^d[A] \cap B(\bo_d, r)$. It remains to apply Proposition \ref{prop:finAttr}.
\end{proof}

\subsection{Auxiliary lattice and division with remainder with respect to it}

Let $(A,\bdd)$ be a digit system. A necessary condition for $(A,\bdd)$ to have the finiteness property is the fact that $\bdd$ contains a complete set of residue class representatives of the group $\Z^d[A]/A\Z^d[A]$. Thus it is desirable to shed some light on this group. While in the case of an integer matrix $A$ we have $|\det A|=|\Z^d[A]/A\Z^d[A]|$, for a rational matrix $A$ the group $\Z^d[A]/A\Z^d[A]$ and its size is not so easy to. In order to get information on $\Z^d[A]/A\Z^d[A]$, one must obtain the base matrix $L \in \Q^{d \times d}$ for the rational lattice $A\Z_k^d[A]=L\Z^d$ with $k=k_A$ described in Lemma \ref{lem:Res}. Unfortunately, no satisfactory representation theory for the nested lattices $\Z^d_k[A]$ in $\Z^d[A]$ is yet available for the determination of the required index $k=k_A$ in Lemma \ref{lem:Res} where the nested chain of lattices stabilizes. Another drawback is that, in general, $\Z^d$ is not be preserved when doing division with remainder modulo $A\Z^d[A]$. In order to deal with these issues, in this subsection we introduce an auxiliary lattice and a division with remainder with respect to this lattice, that eventually can be used for the computation of digital expansions in $\Z^d[A]$. Indeed, it turns out that complete sets of residue classes are easy to deal with in this auxiliary lattice. Moreover, they contain complete sets of residue classes of  $\Z^d[A]/A\Z^d[A]$. 

\begin{definition}\label{def:Aux}
The lattice $\Z^d \cap A\Z^d$ will be referred as \emph{the auxiliary lattice} of the matrix $A \in \Q^{d \times d}$. 
\end{definition}

If $A$ is invertible, $\Z^d \cap A\Z^d$ is a full-rank sublattice of $\Z^d$, hence, in this case  $\Z^d/\left(\Z^d \cap A\Z^d\right)$ is a finite group. Furthermore, we have the following result.

\begin{lemma}\label{lemAuxGroup}
Let $A\in \Q^{d\times d}$ be given.
Then $\Z^d[A]/A\Z^d[A]$ is isomorphic to a subgroup of $\Z^d \big/ \left(\Z^d \cap A\Z^d\right)$. 
\end{lemma}

\begin{proof}[Proof of Lemma \ref{lemAuxGroup}]
By applying the $3$rd Isomorphism theorem to nested modules $\Z^d \cap A\Z^d \subset \Z^d \cap A\Z^d[A] \subset \Z^d$, we obtain
\[
\frac{\Z^d \big/ \left(\Z^d \cap A\Z^d\right)}{\left(\Z^d \cap A\Z^d[A]\right) \big/ \left(\Z^d \cap A\Z^d\right)} \simeq \frac{\Z^d}{\Z^d \cap A\Z^d[A]} \simeq \Z^d[A]/A\Z^d[A],
\]
where the last isomorphism comes from Lemma \ref{lem:Res}. 
\end{proof}

An advantage in using the auxiliary lattice $\Z^d \cap A\Z^d$ 
for arithmetics comes from the fact that its base matrix can be computed directly from the base matrix $A$ using its Smith Normal Form, see Section~\ref{sec:theorySNF}. We will now develop a special kind of division with remainder based on this auxiliary lattice and its residue group. The first step of this construction is to restrict our division to $\Z^d$ according to the following definition.

\begin{definition}\label{def:Restr}
Let $\bdd \subset \Z^d$ be a finite digit set that contains a complete set of coset representatives of $\Z^d/(\Z^d \cap A\Z^d)$. \emph{A restricted digit function} is a mapping $\bd_{\mathrm{r}}: \Z^d \to \bdd$ with the homomorphism property $\bd_{\mathrm{r}}(\bx) \equiv \bx \pmod{\Z^d \cap A\Z^d}$. To this restricted setting we associate the dynamical system  $\Phi_{\mathrm{r}}: \Z^d \to \Z^d$ given by $\Phi_{\mathrm{r}}(\bx) := A^{-1}(\bx - \bd_{\mathrm{r}}(\bx))$.
\end{definition}

The attractor $\aa_{\Phi_{\mathrm{r}}} \subset \Z^d$ of the dynamical system $\Phi_{\mathrm{r}}$ is defined analogously to the attractor of~$\Phi$.

In the second step of our construction we extend this restricted division procedure to the full module $\Z^d[A]$. For any $\bx \in \Z^d[A]$, there always exists a smallest $k \in \N$, such that $\bx \in \Z_k^d[A]$. That is, there exists a shortest expansion of the form
\begin{equation}\label{eq:minExp}
\bx = \bx_0 + A\bx_1 + \dots + A^{k-1}\bx_{k-1},
\end{equation}
with $\bx_0,\ldots,\bx_{k-1} \in \Z^d$. In general, such expressions (even the shortest ones) are not unique: for instance, one can modify $\bx_0$ by adding any element $\bv \in \Z^d \cap A\Z^d$ to it, and then subtracting $A^{-1}\bv$ from $\bx_1$.
However, we re-impose the uniqueness by introducing arbitrary (lexicographic) order $\mathcal{O}$ on $k$-tuples of vectors from $\Z^d$, for $k \in \N$, and then always picking the  expansion \eqref{eq:minExp} whose set of vectors $(\bx_0, \dots, \bx_{k-1})$ is the smallest w.r.t.\ this order. Note that the definition of such an order is always possible, since the set of of all such representations $\eqref{eq:minExp}$ is set-isomorphic to the countable set $\bigcup_{k=1}^{\infty}\Z^{d \times k}$. 
This order enables us to define a dynamical system $\Psi: \Z^d[A] \to \Z^d[A]$. Indeed, let 
\eqref{eq:minExp} be the minimal expansion of $\bx \in  \Z^d[A]$ w.r.t.\ the order $\mathcal{O}$. Then $\Psi$ is given by
\begin{equation}\label{efExtDiv}
\Psi(\bx) := A^{-1}(\bx - \bd_\mathrm{r}(\bx_0)) = \left(\Phi_{\mathrm{r}}(\bx_0)+\bx_1 \right) + A\bx_2 + \dots + A^{k-2}\bx_{k-1}.
\end{equation}
(We note that the expression on the right hand side of \eqref{efExtDiv} is not necessarily the shortest possible expansion of $\Psi(\bx)$.)

The attractor $\aa_{\Psi}$ of the dynamical systems $\Psi$, respectively, is defined analogously to the attractor of~$\Phi$.

Let $A\in \Q^{d\times d}$ be given.
The following lemma shows that it suffices to look at the dynamical system $\Phi_{\mathrm{r}}$ acting on $\Z^d$ in order to construct a digit system in the larger space $\Z^d[A]$ that enjoys the finiteness property.

\begin{lemma}\label{lem:attrExt}
Let $A\in \Q^{d\times d}$ and let $\bdd \subset \Z^d$ be a finite set that contains a complete set of coset representatives of $\Z^d/(\Z^d \cap A\Z^d)$. 
If $\Phi_\mathrm{r}$ in $\Z^d$ has a finite attractor $\aa_{\Phi_\mathrm{r}}$, then $\aa_{\Phi_\mathrm{r}}=\aa_{\Psi}$. In particular, if $\Phi_\mathrm{r}$ is ultimately zero in $\Z^d$, then $\Psi$ is  ultimately zero in $\Z^d[A]$ and the digit system $(A, \bdd_\mathrm{r} \cup \aa_{\Phi_\mathrm{r}})$
has the finiteness property.
\end{lemma}

\begin{proof}[Proof of Lemma \ref{lem:attrExt}]
As $\Psi$ maps $\Z_j^d[A]$ to $\Z_{j-1}^d[A]$, for every $\bx \in \Z^d[A]$, there exists $k=k(\bx)$, such that $\Psi^k(\bx) \in \Z^d$. Thus, by the definition of $\Psi$ in \eqref{efExtDiv} we gain $\Psi^{n+k}(\bx) = \Phi_\mathrm{r}^n(\Psi^k(\bx))$ for each $n\ge 0$. Thus, one eventually ends up with an expansion of an integral vector $\Psi^k(\bx) \in \Z^d$. This proves that $\aa_{\Phi_\mathrm{r}}=\aa_{\Psi}$. The remaining assertions immediately follow from this identity.
\end{proof}

By using $\Z^d$, $\Phi_\mathrm{r}$,  $\dd_\mathrm{r}$, and $\Z^d \cap A\Z^d$ in place of $\Z^d[A]$, $\Phi$, $\dd$, and $A\Z^d[A]$ in the proof of Proposition \ref{prop:finExp} we obtain immediately that, for every expanding matrix $A \in \Q^{d \times d}$, dynamical system $\Phi_\mathbf{r}$ has a finite attractor $\aa_{\Phi_\mathbf{r}}$ in $\Z^d$. By Lemma \ref{lem:attrExt}, the same set $\aa_{\Phi_{\mathrm{r}}}$ then is a finite attractor of the extended remainder division $\Psi$ in $\Z^d[A]$. In a similar way, one can replace $\Phi$ with $\Psi$ in Proposition \ref{prop:finExp} and prove that the periodicity and finiteness properties of $\Psi$ are completely analogous to those of $\Phi$. Thus, in most cases $\Psi$ can replace the classical remainder division $\Phi$ when performing arithmetics in $\Z^d[A]$.

In the same way as in \eqref{eq:xgenphi} we define finite expansions of $\Z^d$ generated by $\Phi_\mathrm{r}$ as well as finite expansions of $\Z^d[A]$ generated by $\Psi$. Using this terminology we get the following corollary.

\begin{corollary}\label{col:RestrExt}
If $A \in \Q^{d \times d}$ is expanding, then one can always find a digit set $\bdd \subset \Z^d$, such that the following assertions hold.
\begin{enumerate}[a)]
    \item every $\bx \in \Z^d$ has a finite expansion in $\bdd[A]$ generated by $\Phi_\mathrm{r}$.
    \item every $\bx \in \Z^d[A]$ has a finite expansion generated by $\Psi$. This implies that $(A, \bdd)$ has the finiteness property. 
\end{enumerate}
\end{corollary}

Thus, the dynamical system $\Psi$ allows to do radix expansions in $\Z^d[A]$, while preserving $\Z^d$. This feature is extremely handy in constructing ``twisted sums'' of digit systems in a way that avoids ``mixing'' the arithmetics of $\Z^d[A]$ and $\Z^d[B]$, for different matrices $A \ne B$. We will come back to this in Section~\ref{sec:twist}.

\section{The finiteness property for generalized rotations}\label{secPerGenRot}

A matrix in $\R^{d \times d}$ that is similar to some orthogonal matrix in $\R^{d \times d}$ is called \emph{a generalized rotation}. These matrices will be of importance when we have to deal with eigenvalues of modulus one in general matrices. In this section we will show how to get small digit sets that ensure the finiteness property for the rational matrices that are generalized rotations. 

\subsection{Compactness criterion for solutions to a norm inequality}\label{secConvex}
Equip the space $\R^d$ with the usual Euclidean norm $\norm{\cdot}$ and let $\sss =\left\{\bv_1, \bv_2, \dots, \bv_s \right\}$ be a finite subset of $\R^d$. Recall that  \emph{the cone} generated by vectors from $\sss$ is defined as
\[
\cone{\sss} = \left\{t_1\bv_1 + t_2\bv_2 + \dots + t_s \bv_s: t_j \geq 0, 1 \leq j \leq s\right\}.
\]
It will turn out that the condition $\cone{\sss}=\R^d$ will be an important property of certain collections of digits $\sss$ for systems whose bases are generalized rotations. One can show that $\cone{\sss}=\R^d$ is equivalent to the fact that $\bo_d$ is an inner point of \emph{the convex hull} of $\sss$, which is defined as
\[
\conv{\sss} = \left\{t_1\bv_1 +t_2\bv_2 + \dots + t_s \bv_s: t_j \geq 0, t_1 + \dots + t_s \leq 1, 1 \leq j \leq s\right\}.
\]
%
%
%
%
%

We will give two simple examples of sets that satisfy this property. These examples will be useful latter.

\begin{example}\label{ex1}
Let $\boe_1$, $\dots$, $\boe_d$ be a standard basis of $\R^d$. Define the set $\sss = \{\bv_1, \dots, \bv_{d+1}\}$ by
\[
\bv_j := \begin{cases}
				\boe_j,& \text{ for } 1 \leq j \leq d\\
				-\boe_1 - \dots - \boe_d,& \text{ for } j = d+1.\\
		\end{cases}
\]
Since $\{\bv_j, 1 \leq j \leq d\}$ is the standard basis of $\R^d$, and
\[
-\bv_k = \sum_{\substack{j=1\\j \ne k}}^{d+1}\bv_j,
\] one sees immediately that $\cone{\sss} = \R^d$. 
\end{example}

\begin{example}\label{ex2} Let $\sss \subset \R^d$ be finite, $\bv \in \R^d$ arbitrary fixed vector and suppose that the matrix $M \in \R^{d \times d}$ is invertible. If $\cone{\sss}=\R^d$, then, for every sufficiently large $\la >0$, $\cone{\la M\sss +\bv}=\R^d$.
\end{example}

We need to restate previous definitions in terms of matrices. Recall that, for any two vectors $\bv, \bw$  with real entries and equal dimensions, one writes $\bv \leq \bw$ if every coordinate of $\bv$ is less or equal than the corresponding coordinate of $\bw$: $v_k \leq w_k$. In particular, $\bw \geq \bo_d$ means that all coordinates of $\bw$ are non-negative. Let $M \in \R^{s \times d}$ be the matrix whose rows consist of vectors from $\sss$, that is $M^T := \left(	\bv_1\,\bv_2\,\dots \bv_s\right)$.
Then one can rewrite the definition of the cone as 
\[
\cone{\sss} =\{ M^T\bu: \bu \in \R^s, \bu \geq \bo_s\}. 
\]
Now we can prove an important criterion on the compactness of the set solutions to the matrix inequality in terms of the dual condition on $M^T$.
\begin{lemma}\label{lemDual}
Let $M \in \R^{s \times d}$ be a real matrix with row vectors $\sss=\{\bv_1, \dots, \bv_s\} \in \R^d$ and fix a vector $\bb \in \R^s$, with $\bb \geq \bo_s$. Then the subset $\vv := \{\bx \in \R^d: M\bx \leq \bb\}$ of $\R^d$ is compact if and only if $\cone{\sss}=\R^d$ (equivalently $\bo_d \in \text{Int}\left(\conv{\sss}\right)$.
\end{lemma}

\begin{proof}[Proof of Lemma \ref{lemDual}]``$\Leftarrow$'': Suppose that $\cone{\sss}=\R^d$. Obviously, $\vv$ is closed in $\R^d$ and $\bo_d \in \vv$. Let $\bx \in \vv$ be arbitrary. Since $\cone{\sss}=\R^d$, for each standard basis vector $\boe_j \in \R^d$, $1 \leq j \leq d$, one can find vectors $\by_j \geq \bo_s, \bz_j \ge \bo_s$ in $\R^s$, such that $\by_j^TM = \boe_j^T$, $\bz_j^TM = -\boe_j^T$. The left-multiplication on left and right sides of $M\bx \leq \bb$ by $\by_j$ and by $\bz_j$, respectively,  yields
\[x_j := \boe_j^T \bx = (\by_j^T M)\bx \leq \by_j^T\bb, \quad \text{ and } \quad -x_j = \boe_j^T (-\bx) = (\bz_j^T M)\bx \leq \bz_j^T\bb.
\]
As the coordinates $x_j$ of $\bx$ are bounded by $-\bz_j^T\bb \leq x_j \leq \by_j^T\bb$, the set $\vv$ must be compact.

\noindent ``$\Rightarrow$'': By contradiction, suppose that $\vv$ is compact but $\sss$ has no proper enclosure. Since $\cone{\sss} \ne\R^d$, there exists $\bc \in \R^d$, such that the linear problem $M^T\by = \bc$ has no solution $\by \in \R^s$ with $\by \geq \bo_s$. Then, by \emph{Farkas lemma} (named after \cite{Far}, in modern form, stated in \cite{GaKuTu}*{Lemma~1 on p.~318}), the dual linear problem $M \bx \geq \bo_s$ has a non-zero solution $\bx = \bn \in \R^d$, such that $\bn^T\bc < 0$, (geometrically, a hyper-plane with the normal vector $\bn \in \R^d$ through the origin separates $\bc$ from $\cone{\sss}$). Since $\bb \geq \bo_s$,  this means one can find arbitrarily large solutions  to the system $M\bx \leq \bb$ by taking $\bx = -\la \bn \in \vv$ and arbitrarily large $\la > 0$. This contradicts the compactness of $\vv$.
\end{proof}
An easy consequence of Lemma \ref{lemDual} is the following:

\begin{remark}\label{rem:rs} If the set of solutions $\vv := \{\bx \in \R^d: M\bx \leq \bb\}$ is compact, then the set $\sss=\text{Columns}(M^T) \subset \R^d$ contains at least $s \geq d+1$ vectors, and some $d$ of these are linearly independent among each other.
\end{remark}

Lemma \ref{lemDual} is the main ingredient of the following result, which will have very important application later. 

\begin{corollary}\label{corComp} Suppose that $\sss \subset \R^d$ is finite. Then the set of vectors $\bx \in \R^d$ which, for every $\bv \in \sss$, satisfy
\[
\norm{\bx - \bv} \geq \norm{\bx}
\] is compact if and only if $\cone{\sss} =\R^d$.
\end{corollary}
\begin{proof}[Proof of Corollary \ref{corComp}] The inequality is equivalent to
$\norm{\bx}^2 - 2\bv^T\bx+\norm{\bv}^2 \geq \norm{\bx}^2$, or,
$\bv^T\bx \leq \frac12\norm{\bv}^2$, for each $\bv \in\sss$. Suppose that $\sss = \{\bv_1, \dots, \bv_s\}$ and let $M \in \R^{s \times d}$ be the matrix whose rows consist of vectors from $\sss$, that is $M^T := \left(	\bv_1\,\bv_2\,\dots \bv_s\right)$.
One obtains a system of linear inequalities
$M\bx \leq \bb$ where the non-negative vector $\bb \in \R^s$ is defined by $b_j := 1/2\norm{\bv_j}^2$. Now apply Lemma \ref{lemDual}.
\end{proof}

\subsection{Generalized rotation and its invariant norm}\label{secGenRot}
We now give the formal definition of the main objects of the present section.

\begin{definition}[Generalized rotation]\label{defRot}
A real square matrix $M \in \R^{d \times d}$ is called {\em generalized rotation}, if $M$ is similar to some orthogonal matrix $R$.
\end{definition}

It is known that matrices $M \in \R^{d \times d}$ that are diagonalizable over $\C$, with all eigenvalues $\la \in \C$ of absolute value $\abs{\la}=1$ are general rotations; see, for instance 
\cite{Gant}*{[Chapter 9, Section 13, Equations (102)--(106)]}. For a general rotation $M$ there always exists an invertible real transformation $T \in \R^{d\times d}$,  such that $M := TRT^{-1}$ takes a block-diagonal form with $1 \times 1$ blocks $(\pm 1)$ or $2 \times 2$ blocks
\[
\begin{pmatrix}
\cos{\theta} & -\sin{\theta}\\
\sin{\theta} & \cos{\theta}
\end{pmatrix}, \; \theta \in [0, 2\pi)
\] 
that correspond to eigenvalues $\la =e^{i\theta}$. Such $R$ is a composition of mutually orthogonal rotations and is itself an orthogonal matrix, {\it i.e.}, $R^TR = \mathds{1}_d$. Details of such decomposition are outlined in \cite{Gant}*{Chapter~9, Section~13}), equations (112)--(113) therein. Furthermore, orthogonal matrices $R$ can be parametrized using Cayley formulas \cite{Gant}*{Chapter~9, equations (123)--(126)} ({\it cf.} \cites{Rei, LieOsb}). Explicit parametrization formulas in small dimensions $d=3, 4$ are available \cites{Pall, Sho, Crem, duVal, Meb}.   


If $M$ is a general rotation, then there exist an $M$-invariant norm on $\R^d$. Indeed, let $T$ be the aforementioned transformation that brings $M$ to its rotational form $R$. Define
$\norm{\bx}_{M} := \norm{T^{-1}\bx}$, where $\norm{\cdot}$ is the usual Euclidean norm in $\R^d$. Then
\[
\norm{M\bx}_{M} = \norm{T^{-1}M\bx}= \norm{(T^{-1}MT)T^{-1}\bx)} = \norm{RT^{-1}\bx} = \norm{T^{-1}\bx} = \norm{\bx}_{M},
\] since $R$ is orthogonal, and $\norm{R\by}=\norm{\by}$ for any $\by \in \R^d$.

One important fact that we will need in the sequel: Corollary \ref{corComp} of previous Section \ref{secConvex} still holds true if one replaces the Euclidean norm $\norm{\cdot}$ by this new norm $\norm{\cdot}_{M}$.

\begin{corollary}\label{corMinv} Let $\sss \subset \R^d$ be a finite set and suppose that $M \in \R^{d \times d}$ is a generalized rotation. Then the set of vectors $\bx \in \R^d$ which, for every $\bv \in \sss$, satisfy $\norm{\bx - \bv}_{M} \geq \norm{\bx}_{M}$ in $M$-invariant norm $\norm{\cdot}_M$ of $\R^d$ is compact if and only if $\cone{\sss}=\R^d$.
\end{corollary}
\begin{proof}[Proof of Corollary \ref{corMinv}] The $M$-invariant norm inequality reads
$\norm{T^{-1}\bx - T^{-1}\bv} \geq \norm{T^{-1}\bx}$ in Euclidean norm, where $T \in \R^{d \times d}$ is invertible. By continuity, $T^{-1}$ and $T$ both preserve the open inclusion of $\bo_d$ between $\sss$ and $T^{-1}\sss$, as well as the compactness. Hence, Corollary \ref{corComp} applies.
\end{proof}

\subsection{Finiteness property for generalized rotations}

In this section, we assume that $A \in \Q^{d \times d}$ is a generalized rotation, characterized in Definition \ref{defRot} of Section \ref{secGenRot}. Generalized rotations serve as a fundamental building block to construct digit systems with the finiteness property for general rational matrix with eigenvalues of absolute value $= 1$.
For general rotation bases, we introduce digit sets with special geometric properties.
\begin{definition}[Good enclosure]\label{defConv}
The digit set $\bdd \subset \Z^d[A]$ is said to have \emph{a good enclosure}, if, for each coset representative  $\bx \in \Z^d[A]/A\Z^d[A]$, the subset
\[
\bdd(\bx) := \left\{\bd \in \bdd: \bd \equiv \bx \pmod{A\Z^d[A]}\right\}
\]
satisfies $\cone{\bdd(\bx)}=\R^n$, or, equivalently: the origin $\bo_d$ always lies in the interior of the convex hull of $\bdd(\bx)$ for any given residue class modulo $\pmod{A\Z^d[A]}$. 
\end{definition}

If $A$ is general rotation, then $A^{-1}$ is also general rotation. Pick an $A^{-1}$-invariant norm $\norm{\cdot}_{A^{-1}}$ in $\R^d$, defined in Section~\ref{secGenRot}. Enumerate the digits
$\bdd =\{ \bd_1, \dots, \bd_k\}$ and define \emph{the digit function} $\bd: \Z^d[A] \to \bdd$ by
\begin{equation}\label{eq:defDigitF}
\bd(\bx) := \bd_j \in \bdd, \text{ such that }
\begin{cases}
	\bd_j \equiv \bx \pmod{A\Z^d[A]},\\
	\bd_j \text{ minimizes} \norm{\bx - \bd_j}_{A^{-1}} \text{ in } \bdd,\\
	\text{the index $j$ is the smallest possible}.
\end{cases}  
\end{equation}
The smallest index condition is to impose uniqueness in cases where the norm-minimum is achieved for several $\bd_j \in \bdd(\bx)$. The quotient function $\Phi: \Z^d[A] \to \Z^d[A]$ used in the division with remainder for the digit sets with good enclosures is defined as in \eqref{eq:defRemDiv}.

Now comes the most important result of this section.

\begin{theorem}\label{thmPeriodic} Let $A \in \Q^{d \times d}$ be a generalized rotation. If $\bdd$ has a good enclosure, then the remainder division $\Phi: \Z^d[A] \to \Z^d[A]$ with the digit function from \eqref{eq:defDigitF} has a finite attractor set $\aa_\Phi \subset \Z^d[A]$. In particular, $\Phi$ is ultimately periodic mapping with a finite number of possible smallest periods.
\end{theorem}

\begin{proof}[Proof of Theorem \ref{thmPeriodic}]
First we show that the set $\vv := \left\{ \bx \in \Z^d[A]: \norm{\Phi(\bx)}_{A^{-1}} \geq \norm{\bx}_{A^{-1}} \right\}$ is bounded in $\R^d$. By $A^{-1}$-invariance, $\norm{\Phi(\bx)}_{A^{-1}} = \norm{\bx - \bd(\bx)}_{A^{-1}}$. For each $\bx \in \vv$, $\bdd(\bx) \ne \varnothing$ by Definition \ref{defConv} and $\norm{\bx - \bd_j}_{A^{-1}} \geq \norm{\bx}_{A^{-1}}$
holds for each digit $\bd_j \in \bdd(\bx)$ by \eqref{eq:defDigitF}. Since $\bdd(\bx)$ has good enclosure, by Corollary \ref{corMinv}, the set of vectors $\bx \in \R^d$ that satisfy the last stated norm inequality for each $\bd_j \in \bdd(\bx)$ must be compact. Therefore $\vv$ must be bounded, and we can define the critical radius
$r := r(\vv, \bdd) := \sup_{\bx \in \vv}\norm{\bx}_{A^{-1}}+\max_{\bd \in \bdd}\norm{\bd}
_{A^{-1}}$.

By Lemma \ref{lem:RemDiv}, there exists $l \in \N$, such that the orbit $\{\bx_n := \Phi^n(\bx)\}_{n=1}^{\infty}$ of every point $\bx \in \Z^d[A]$ eventually reaches and then stays in the lattice $\Z_l^d[A]$. Since $\Z_l^d[A]$ is a lattice, one cannot have $\norm{\bx_{n+1}} < \norm{\bx_n}$ for every $n > n_0$. Therefore, $\norm{\bx_{n+1}} \geq \norm{\bx_n}$ occurs infinitely many times. For such $n$, one must have $\bx_n \in \vv \cap \Z_l^d[A]$. At the next step after such occurrence, we have
\[
\norm{\bx_{n+1}}_{A^{-1}}=\norm{\Phi(\bx_n)}_{A^{-1}}=\norm{\bx_n-\bd(\bx_n)}_{A^{-1}}\leq \norm{\bx_n}_{A^{-1}}+\norm{\bd(\bx_n)}_{A^{-1}} \leq r(\vv, \bdd).
\]
In subsequent iterations, we have either $\bx_{j+1} \in \vv$ again, or the $\norm{\bx_{j+1}}_{A^{-1}} < \norm{\bx_{j}}_{A^{-1}}$. Therefore, $\Phi^n(\bx)$ ultimately enters the ball $B_{A^{-1}}(\bo, r)$ and then never leaves it. Since $\Z_l^d[A]$ is a lattice, $B_{A^{-1}}(0, r) \cap \Z_l^d[A]$ must be finite. Hence, $\Phi$ has a finite attractor set $\aa_{\Phi} \subset B_{A^{-1}}(0, r) \cap \Z_l^d[A]$. The last statement of Theorem \ref{thmPeriodic} follows from the fact that $\Phi^{n+1}(\bx)$ is uniquely determined by $\Phi^n(\bx)$.
\end{proof}

\begin{corollary}\label{corSmallD}
For every general rotation $A \in \Q^{d \times d}$, there exists a digit set $\bdd \subset \Z^d[A]$ of the size $\#\bdd = (d+1)\left[\Z^d[A]:A\Z^d[A]\right]$, such that the division mapping $\Phi: \Z^d[A] \to \Z^d[A]$ is ultimately periodic and has finite attractor. In particular, one can find such digit sets in $\Z^d$.
\end{corollary}
\begin{proof}[Proof of Corollary \ref{corSmallD}] Let $\rr$ be a complete set of residue class representatives of the residue class ring $\Z^d/(\Z^d \cap A\Z_k^d[A])$ and take $\sss = \{\boe_1, \dots, \boe_d, -\boe_1-\dots-\boe_d\}$ from Example \ref{ex1} of Section \ref{secConvex}. For every sufficiently large integer $s \in \N$, that is divisible by all denominators of the rational entries of the matrix $A$, the digit set $\bdd := sA\sss + \rr \subset \Z^d$ contains all residue classes $\pmod{A\Z^d[A]}$ and satisfies the conditions of Theorem~\ref{thmPeriodic}, as explained in Example \ref{ex2} of Section~\ref{secConvex}. Now apply Theorem~\ref{thmPeriodic}.
\end{proof}

\begin{corollary}\label{cor:finRot}
For every generalized rotation $A \in \Q^{d \times d}$, there exist digit sets $\bdd \subset \Z^d[A]$ (and even $\bdd \subset \Z^d$), such that the digit system $(A, \bdd)$ has the finiteness property in $\Z^d[A]$.
\end{corollary}
\begin{proof}[Proof of Corollary \ref{cor:finRot}] Consider the digit set $\bdd'$ used in Corollary \ref{corSmallD} and apply the second part of Proposition \ref{prop:finAttr} to $\bdd := \bdd' \cup \aa_{\Phi}$.
\end{proof}

\begin{remark}\label{rem:rotRestr}
All results of Section \ref{secPerGenRot} for generalized rotations $A \in \Q^{d \times d}$ remain true for $\Phi_{\mathrm{r}}$ in $\Z^d$ with respect to auxiliary lattice $\Z^d \cap A\Z^d$, provided that one defines the good arithmetical enclosure property in $\Z^d$ $\pmod{\Z^d \cap A\Z^d}$ instead of $\pmod{A\Z^d[A]}$ in Definition \ref{defConv} and  Eq.\eqref{eq:defDigitF}, and substitutes $\Z^d$ in place of $\Z^d[A]$ where appropriate.
\end{remark}

\section{Twisted digits systems for hypercompanion matrices}\label{sec:twist}

In the present section we show how we can build up number systems with property (F) from simple building blocks. We develop a certain version of a direct sum that generalizes the notion of \emph{simultaneous digit systems} previously considered by K\'ovacs \cites{Kov7, Kov8, Kov9}; see also \cite{SSTW}*{Section~3} and \cite{Woe1} for related \emph{products of number systems} in polynomial rings. This will enable us to construct digit sets with Property (F) for the hypercompanion block matrices.

First, let us recall the notations of 'the direct sum' of vectors and matrices. For two vectors $\bx \in \Q^m$ and $\by \in \Q^n$, their direct sum is a block vector $\bx \dsum \by := \vectwo{\bx}{\by} \in \Q^{m+n}$. This notation extends to the sets of vectors $\mathcal{S} \subset \Q^m$, $\mathcal{T} \subset \Q^n$ by
\[
\mathcal{S} \dsum \mathcal{T} := \{ \bx \dsum \by: \bx \in \mathcal{S}, \by \in \mathcal{Y}\},
\] for instance, $\Q^m \dsum \Q^n = \Q^{m+n}$.
Secondly, the direct sum of two matrices, say, $A \in \Q^{m \times m}$ and $B \in \Q^{n \times n}$ is defined as the block matrix
\[
A \dsum B = \begin{pmatrix}
                A & O\\
                O & B.
            \end{pmatrix}
\]
\subsection{Semi--direct (twisted) sums of digit systems}\label{subsec:twistsum} Let $A$, $B$ be two invertible rational matrices of dimensions $m\times m$, $n \times n$, respectively, let $O$ be the $m \times n$ zero-matrix, and $C$ be an $n \times m$ integer matrix. By block multiplication we see that the partitioned matrices
\[
M =\begin{pmatrix}
    A & O\\
    C & B 
\end{pmatrix}, \qquad
M^{-1} =\begin{pmatrix}
    A^{-1}         & O\\
    -B^{-1}CA^{-1} & B^{-1}\\ 
    \end{pmatrix},
\]
are inverse to each other. We will call the partioned matrix $M$ \emph{the sum of $A$ and $B$, twisted by $C$} (a semi--direct sum, or a twisted sum  for short). This will be abbreviated as $M := A \oplus_C B$.

Suppose that $\bdd_A \in \Z^m$ and $\bdd_B \in \Z^n$ are the digit sets satisfying $\Z^m/(\Z^m \cap A\Z^m) \subset \bdd_A$ and $\Z^n/(\Z^m \cap B\Z^m) \subset \bdd_B$. This containment can be strict: digit sets are allowed to be larger than residue sets.

Associated with $(A, \bdd_A)$ and $(B, \bdd_B)$ are the corresponding \emph{restricted digit functions} together with the dynamical systems 
\begin{alignat}{5}
\bd_{\mathrm{r},A}:&&\Z^m \to \bdd_A,& \quad \bd_{\mathrm{r},A}(\bx) \equiv \bx \pmod{\Z^d \cap A\Z^d}, \\
\Phi_{\mathrm{r},A}:&& \Z^m \to \Z^m,& \quad \Phi_{\mathrm{r},A}(\bx) = A^{-1}(\bx - \bd_{\mathrm{r},A}(\bx)),
\\
\bd_{\mathrm{r},B}:&& \Z^n \mapsto \bdd_{B},& \quad \bd_{\mathrm{r},B}(\by) \equiv \by \pmod{\Z^d \cap B\Z^d}, \\
\Phi_{\mathrm{r},B}:&& \Z^n \mapsto \Z^n,& \quad \Phi_{\mathrm{r},B}(\by) = B^{-1}(\by - \bd_{\mathrm{r},B}(\by)),
\end{alignat}
for $\bx \in \Z^m$ and $\by \in \Z^n$. They perform division with remainder in the lattices $\Z^m$ and $\Z^n$, respectively; see Definition \ref{def:Restr} in Section \ref{secArithm}.

\smallskip
We will build the corresponding digit set and restricted remainder division by $M$ in $\Z^m \dsum \Z^n = \Z^{m+n}$ modulo the sublattice $\Z^{m+n} \cap M\Z^{m+n}$.

To account for the carry from the first component to the second, we first set up a modified digit function $\widetilde{\bd}_{\mathrm{r},B}: \Z^{m+n} \to \bdd_B$ and an associated dynamical system $\widetilde{\Phi}_{\mathrm{r},B}: \Z^{m+n} \mapsto \Z^n$ for the division with remainder $\pmod{\Z^n \cap B\Z^n}$ on the second component of $\bz = \bx \oplus \by$ by
\[
\widetilde{\bd}_{\mathrm{r},B}(\bz) := \bd_{\mathrm{r},B}\left(\by-C\cdot \Phi_{\mathrm{r},A}(\bx)\right), \qquad
\widetilde{\Phi}_{\mathrm{r},B}(\bz) := B^{-1}\left(\by-C\cdot \Phi_{\mathrm{r},A}(\bx) - \widetilde{\bd}_{\mathrm{r},B}(\bz)\right).
\]
The function $\widetilde{\bd}_{\mathrm{r},B}(\bz)$ is well-defined, since $C$ is integral and its dimensions $n \times m$ match with the dimensions $m \times 1$ of $\Phi_{\mathrm{r},A}(\bx)$. The function $\widetilde{\Phi}_{\mathrm{r},B}(\bz)$ is also well-defined, since $\by - C \cdot \Phi_{\mathrm{r},A}(\bx) - \widetilde{\bd}_{\mathrm{r},B}(\bz)\equiv \bo\pmod{\Z^n \cap B\Z^n}$, so the multiplication by $B^{-1}$ produces a vector in $\Z^n$.

We are now in the position to define \emph{the twisted restricted division with remainder by $M$}. Let \emph{the twisted digit function} $\bd_{\mathrm{r},M} :\Z^{m+n} \to \bdd_A \dsum \bdd_B$ be given by
\begin{equation}\label{eq:def_dtwist}
 \bd_{\mathrm{r},M}(\bz) := \left((\bd_{\mathrm{r},A} \circ \pi_A)) \oplus \widetilde{\bd}_{\mathrm{r},B}\right)(\bz) =
\begin{pmatrix}
\bd_{\mathrm{r},A}(\bx)\\
\widetilde{\bd}_{\mathrm{r},B}(\bz)
\end{pmatrix}=
\begin{pmatrix}
\bd_{\mathrm{r},A}(\bx)\\
\bd_{\mathrm{r},B}\left(\by - C\cdot \Phi_{\mathrm{r},A}(\bx)\right)
\end{pmatrix},
\end{equation}
where $\pi_A: \Z^{m+n} \mapsto \Z^m$ is simply a projection of the first component $\pi_A\vectwo{\bx}{\by} = \bx$. It will be abbreviated as $\bd_{r, M} = \twist{\bd_{r, A}}{\bd_{r,B}}{C}$.

Then we define the associated dynamical system in the usual way by setting
\[
\Phi_{\mathrm{r},M}: \Z^{m+n} \to \Z^{m+n}, \qquad
\Phi_{\mathrm{r},M}(\bz) := M^{-1}(\bz - \bd_{\mathrm{r},M}(\bz)).
\]
This construction will be abbreviated as $\Phi_{\mathrm{r},M}:=\twist{\Phi_{r, A}}{\Phi_{\mathrm{r}, B}}{C}$. 
One verifies easily that
\[
\Phi_{\mathrm{r},M}(\bz) =
\begin{pmatrix}
A^{-1}          & O\\
-B^{-1}CA^{-1}  & B^{-1}
\end{pmatrix}
\vectwo{\bx-\bd_{\mathrm{r},A}(\bx)}{\by-\widetilde{\bd}_{\mathrm{r},B}(\bz)}=
\]
\[
= \begin{pmatrix}
A^{-1}(\bx - \bd_{\mathrm{r},A}(\bx))\\
B^{-1}\left(-CA^{-1}(\bx - \bd_{\mathrm{r},A}(\bx))+\by-\widetilde{\bd}_{\mathrm{r},B}(\bz)\right)
\end{pmatrix} =
\]
\[
=\begin{pmatrix}
\Phi_{\mathrm{r},A}(\bx)\\
\widetilde{\Phi}_{\mathrm{r},B}(\bz)
\end{pmatrix} =
\left((\Phi_{\mathrm{r},A} \circ \pi_A) \oplus \widetilde{\Phi}_{\mathrm{r},B}\right)(\bz).
\]
Thus, $\Phi_{\mathrm{r},M}$ is defined correctly. Since $\bz = \bd_{\mathrm{r},M}(\bz)+M\Phi_{\mathrm{r},M}(\bz)$, we obtain that $\bz \equiv \bd_{\mathrm{r},M}(\bz) \pmod{\Z^{m+n} \cap M\Z^{m+n}}$. This last congruence implies that the mapping $\bz \mapsto \bd_{\mathrm{r},M}(\bz) \pmod{\Z^{m+n}/(\Z^{m+n} \cap M\Z^{m+n})}$ is a homomorphism from $\Z^{m+n}$ to $\Z^{m+n}/(\Z^{m+n} \cap M\Z^{m+n})$. In particular, $\bdd_{A} \dsum \bdd_{B}$ must contain all the representatives of $\Z^{m+n} / (\Z^{m+n} \cap M\Z^{m+n})$. By \eqref{eq:def_dtwist}, $\bd_{\mathrm{r},M}(\bz)$ attain all possible values in $\bdd_A \dsum \bdd_B$ as $\bx$ and $\by$ run through $\Z^m/(\Z^m \cap A\Z^m)$ and $\Z^n/(\Z^m \cap B\Z^m)$, respectively.

\begin{lemma}\label{lemTwist}
If $\bo_m \in \bdd_A$, $\bo_n \in \bdd_B$ and attractors of $\Phi_{\mathrm{r},A}$ and $\Phi_{\mathrm{r},B}$ consist only of the zero vectors in $\Z^m$ and $\Z^n$, respectively, then $\Phi_{\mathrm{r},M} = \twist{\Phi_{\mathrm{r},A}}{\Phi_{\mathrm{r},B}}{C}$ has attractor $\{\bo_{m+n}\}$ in $\Z^{m+n}$. In this case, the digit system $(M, \bdd_A \dsum \bdd_B)$ in $\Z^{m+n}[M]$, where $M=\twist{A}{B}{C}$, has the finiteness property.
\end{lemma}
\begin{proof}[Proof of Lemma \ref{lemTwist}]
Notice that
\[
\Phi_{\mathrm{r},M}^k(\bz) =
\vectwo{\Phi_{\mathrm{r},A}^k(\bx)}{\widetilde{\Phi}_{\mathrm{r},B}(\Phi_{\mathrm{r},M}^{k-1}(\bz))} =
\vectwo{\bo_m}{\widetilde{\Phi}_{\mathrm{r},B}(\Phi_{\mathrm{r},M}^{k-1}(\bz))}
\]
for $k \geq k_0 = k_0(\bx)$. When the first component $\bx$ of $\bz$ is $\bo_m$, then  $\Phi_{\mathrm{r},M}$ reduces to $\Phi_{\mathrm{r},B}$, {\it i.e.}, $d_{\mathrm{r},M}(\bz) = \bd_{\mathrm{r},B}(\by)$ and $\widetilde{\Phi}_{\mathrm{r},B}(\bz) = \Phi_{\mathrm{r},B}(\by)$. Thus,
\[
\Phi_{\mathrm{r},M}^k(\bz) = \vectwo{\bo_m}{\Phi_{\mathrm{r},B}^{k-k_0}(\pi_B \circ \Phi_{\mathrm{r},M}^{k_0}(\bz))}=
\vectwo{\bo_m}{\bo_n}, \text{ for } k \geq k_0 + k_1, k_1 = k_1(\bz).
\]
Thus, $\Phi_{\mathrm{r}, M}$ is ultimately zero in $\Z^{m+n}$. To show that $(M, \bdd_A \dsum \bdd_B)$ has the finiteness property in $\Z^{n+m}[M]$, apply Lemma~\ref{lem:attrExt} with $\Psi_M$ (which is defined in terms of $\Phi_{\mathrm{r},M}$).
\end{proof}

\subsection{Hypercompanion bases} Equipped with the twisted sums of digit systems we can establish the finiteness result for generalized Jordan blocks matrices. Recall that \emph{the companion matrix} $C(f) \in \Q^{d \times d}$ of a monic polynomial
\[
f(x) = x^d + a_{d-1}x^{d-1}+\dots + a_1x +a_0 \in \Q[x]
\] of degree $d := \deg f \geq 1$ is defined by
\begin{equation*}
C(f) = \begin{pmatrix}
0 & 0 & \dots & 0& 0 & -a_0\\
1 & 0 & \dots & 0& 0 & -a_1\\
0 & 1 & \dots & 0& 0 & -a_2\\
\vdots & \vdots & \ddots &\vdots & \vdots & \vdots\\
0 & 0 & \dots & 1& 0 & -a_{d-2}\\
0 & 0 & \dots & 0& 1 & -a_{d-1}\\
\end{pmatrix}.
\end{equation*}
Following the notation in \cite{Per}*{Section~8.9, pp.~161--163}, we define the  \emph{the hypercompanion matrix} $H_k := H_k(f)$ of $f(x)$ of order $k$. For $k=1$, it is simply equal to $H_1:=C(f)$, while, for $k \geq 2$, it is defined by $k \times k$ block structure
\begin{equation}\label{eq:hypercomp}
H_k(f) :=\begin{pmatrix}
C(f)    & O       & \cdots &    O  & O\\
N       &    C(f) & \cdots &    O  & O\\
O       &       N & \ddots &    O  & O\\
\vdots  & \vdots  & \ddots & C(f)  & O\\
O       &       O & \cdots &    N  & C(f)\\
\end{pmatrix}.
\end{equation}
Here, $N \in \Z^{d \times d}$ has $1$ in its top--right corner and all other entries $0$. Clearly, $H_k \in \Q^{kd \times kd}$. In our paper, we use the transposed form of \cite{Per}, due to our choice of matrix-vector multiplication. An alternative form of hypercompanion matrices, where $N^t$ would appear above the principal diagonal is used in Jacobson~\cite{Jac}*{p.~72}. In the literature, the hypercompanion matrices also appear under the names of \emph{the (primary) rational canonical form} \cite{Har}, or  \emph{the generalized Jordan block} \cite{BJN}: in the special case $d=1$, $H_k$ is simply the classical Jordan block with a rational eigenvalue of multiplicity $k$.

\begin{lemma}\label{lemHypercomp}
Let $f \in \Q[x]$ of degree $d$ be monic and irreducible over $\Q$, with the associated hypercompanion matrix $H_k=H_k(f) \in \Q^{n \times n}$ of order $k$ and dimension $dk$. If all the zeros of $f(x)$ in $\C$ are of absolute value $ \geq 1$, then there exists a digit set $\bdd_{H_k}$, containing the zero vector, a restricted digit function $\bd_{\mathrm{r},H_k}: \Z^{dk} \to \bdd_{H_k}$ with associated dynamical system $\Phi_{\mathrm{r},H_k}:\Z^{dk} \to \Z^{dk}$, such that $\Phi_{\mathrm{r},H_k}$ has attractor $\{\bo_{dk}\}$.
\end{lemma}
\begin{proof}[Proof of Lemma \ref{lemHypercomp}]
Consider the case $k=1$. Then $H=H_1(f)$ is the companion matrix of $f$. Since $f$ is irreducible over $\Q$, there are two possibilities: either $f$ has all roots of absolute value $>1$, or all roots of absolute value $=1$. In the first case, $H^{-1}$ is a contraction, and we can take the digit set $\bdd_{H}'$ which is a complete set of coset representatives of the residue class ring $\Z^d/(\Z^d \cap H\Z^d)$ which leads to a finite attractor of $\Phi_{\mathrm{r}, H_1}$. In the second case, $H$ is a generalized rotation, so according to the Remark \ref{rem:rs}, there exists a digit set $\bdd_{H}'$ that contains all coset representatives of $\Z^d/(\Z^d \cap H\Z^d)$ and that has a good enclosure. In both cases, the division mapping $\Phi_{\mathrm{r},H}(\bx) = H^{-1}(\bx - \bd_{\mathrm{r},H}(\bx))$, where  $\bd_{\mathrm{r},H}: \Z^d \to \bdd_{H}'$ satisfies $\bd_{\mathrm{r},H}(\bx) \equiv \bx \pmod{\Z^d \cap H\Z^d}$ is ultimately periodic by Remark~\ref{col:RestrExt} and Remark~\ref{rem:rs}. Now take $\bdd_{H}$ to be the union of $\bdd_{H}'$, the zero digit $\bo_d$ and the points from the attractor of $\Phi_{\mathrm{r}, H}$ in $\Z^d$, and modify the digit function $\bd_{\mathrm{r},H}$ so that it $\bd_{\mathrm{r},H}(\bx)=\bx$ for $\bx=\bo_d$ or $\bx$ in the attractor. Clearly, $\Phi_{\mathrm{r},H}$ has attractor $\{\bo_d\}$. Hence, $(H, \bdd_{H})$ now has  Property (F) in $\Z^d[H]$.

Assume that the lemma is true for all hypercompanion matrices of order $\leq k$, and that we already have the digit sets $\bdd_{H_k}$ and functions $\bd_{\mathrm{r},H_k}$, $\Phi_{\mathrm{r},H_k}$ with the required properties. Notice that, by denoting the ${d\times d(k+1)}$ block matrix $G= (O\;\dots\;O\;N)$ with integer entries, one can write $H_{k+1}$ is a twisted sum $H_{k+1} = \twist{H_k}{H_1}{G}$. By Lemma~\ref{lemTwist}, the dynamical system $\Phi_{\mathrm{r},H_{k+1}} := \twist{\Phi_{\mathrm{r},H_k}}{\Phi_{\mathrm{r},H_1}}{G}$ with the digit set $\bdd_{H_k} \dsum \bdd_{H_1}$ and the digit function $\bd_{\mathrm{r},H_{k+1}} := \twist{\bd_{\mathrm{r},H_k}}{\bd_{\mathrm{r},H_1}}{G}$ has attractor $\{\bo_{d(k+1)}\}$ and, hence, $(H_{k+1}, \bdd_{H_k} \dsum \bdd_{H_1})$ has the finiteness property in $\Z^{d(k+1)}[H_{k+1}]$.
\end{proof}

\section{Applications to the general theory}
\label{sec:general}

In this section we use our theory in order to give new proofs of some general characterization results of Property (F) established in \cite{AkThZa} and \cite{JanThu}. The advantage of these proofs is the fact that they are more algorithmic and allow good control on the size of the digit sets.

In \cite{JanThu}, we have proved the following result.

\begin{theorem}[{Jankauskas and Thuswaldner~\cite{JanThu}}]\label{thm:Main}
Let $A$ be an $d \times d$ matrix with rational entries. There is a digit set $\dd \subset \Z^d[A]$ that makes $(A, \dd)$ a digit system in $\Z^d[A]$ with finiteness property if and only if $A$ has no eigenvalue $\la$ with $\abs{\la} < 1$. The digit set $\bdd$ can even be chosen to be a subset of $\Z^d$.
\end{theorem}

The proof Theorem \ref{thm:Main} in \cite{JanThu} essentially builds on the following theorem.

\begin{theorem}[{Akiyama {\it et al.}~\cite{AkThZa}}]\label{thm:ATZ}
Let $\al \in \C$. Then, there is a finite subset $F$ of $\Z$, such that $\Z[\al] =F[\al]$, if and only if $\al$ is an algebraic number whose conjugates, over $\Q$, are all of modulus one, or all of modulus greater than one.
\end{theorem}

In order to prove Theorem~\ref{thm:Main} in \cite{JanThu}, the main result of \cite{AkThZa}, Theorem~\ref{thm:ATZ}, was translated into the polynomial ring setting and formulated in terms of a finiteness property for the representations of polynomials in the quotient ring $\Z[x]/(f)$, where $f=f(x) \in \Z[x]$. Then, it was shown that this finiteness property can be transferred from these intermediate polynomial rings to companion matrices of powers of irreducible factors of $f(x)$. The last step was to ``assemble'' the finite representations together through the Frobenius normal form of $A$. While our approach in \cite{JanThu} resulted in a very concise proof of Theorem~\ref{thm:Main}, the introduction of intermediate algebraic structures make the practical computation of the radix expansions extremely complicated. First key moment -- the control on the orbit of $A^{-n}\bx$ in the directions of eigenvectors with $\abs{\la}=1$  -- is achieved through the embedding into a more complicated representation space $\R^d \times \prod_{\mathfrak{p}}\Q_{\mathfrak{p}}$ (an open subset of an ad\`ele ring) in the proof of Theorem~\ref{thm:ATZ} in \cite{AkThZa}. The second key moment -- control over the behavior of $A^{-n}\bx$ in the directions that correspond to Jordan blocks of order $m \geq 2$ -- is performed in \cite{JanThu} by convolving the finite representations in the intermediate rings $\Z[x]/(f^m)$. This obscures the dynamics of the radix expansion process. Another side effect is, that the introduction of the intermediate structures and the subsequent process of assembling the finite representations typically blows out the sizes of the digit sets. All factors combined, the proof of Theorem~\ref{thm:Main} in \cite{JanThu} does not translate in a straightforward way into a practical algorithm and makes the computation of radix expansions in $(A, \bdd)$ rather complicated. These drawbacks are no longer there in the proofs of these theorems we will now provide.

\begin{proof}[Proof of Theorem \ref{thm:Main}]
We first prove the sufficiency: if all the eigenvalues of $A$ are of absolute value $\geq 1$, there exists $(A, \bdd)$ with the finiteness property in $\Z^d[A]$.

By \cite{Per}*{Chapter~8, p.~162, Theorem~8.10} there exists an invertible matrix $T \in \Q^{d \times d}$, such that $B=T^{-1}AT$ takes a block diagonal form $B = \oplus_{j=1}^r H_j$, where each block $H_j = H_{k_j}(f_j) \in \Q^{d_j \times d_j}$, $1 \leq j \leq r$, $1 \leq d_j \leq d$, $d_1 + \dots + d_r = d$, is a hypercompanion matrix of order $k_j$ of a monic irreducible polynomial $f_j(x) \in \Q[x]$ that divides the characteristic polynomial $\Phi_{\mathrm{r}}(x)$ of $A$. 
Since for any $t \in \Z$, $(tT)A(tT)^{-1} = TAT^{-1}$, we can assume that $T \in \Z^{d \times d}$. First, we will construct the digit system with the finiteness property for the module $\Z^d[B]$.

By assumption, $A \in \Q^{d \times d}$ has no eigenvalues $\la \in \C$ of absolute value $\abs{\la} < 1$. By the irreducibly of $f_j$, all eigenvalues of $H_j$ are of absolute value $>1$, or all of them are of absolute value $=1$. We consider two possibilities: a) every $H_j$ either is expanding, or has order $k_j=1$; b) some $H_j$ of order $k_j\geq 2$ are not expanding, that is, have only unimodular eigenvalues.

\medskip

\noindent \emph{Case a)} according to Proposition \ref{prop:finExp} (for expanding matrices $H_j$) and by Corollary \ref{cor:finRot} (for generalized rotation matrices $H_j$), there exists digit sets $\bdd_j \subset \Z^{d_j}[H_j]$ (and even $\bdd_j \subset \Z^{d_j}$), such that each digit system $(H_j, \bdd_j)$ has the finiteness property in $\Z^{d_j}[H_j]$; each such digit system is equipped with the dynamical system $\Phi_j: \Z^{d_j}[H_j] \to \Z^{d_j}[H_j]$, each of which has attractor $\{\bo_{d_j}\}$ in $\Z^{d_j}[H_j]$. In order to compensate for the different lengths of radix expansions across all blocks,  we need $\bo_{d_j} \in \bdd_j$ (or simply adjoin the zero vectors to the digit sets). Then we take their cartesian product $\bdd := \dsum_{j=1}^r \bdd_j$. Since $B=\oplus_{j=1}^r{H_j}$, one readily verifies
$\Z^d[B] = \dsum_{j=1}^r{\Z^{d_j}}[H_j]$. By the finiteness property of $(H_j, \bdd_j)$, for each $j$, $1 \leq j \leq r$, $\Z^{d_j}[H_j] = \bdd_j[H_j]$. By collecting terms with equal powers of matrices $H_j^n$ in the same vectors and using padding with zero vectors in each factor, if necessary, when exchanging $\dsum$ and the $\sum$ in radix expansions (see \cite{JanThu}*{(5),(6)}), one obtains
\begin{equation}\label{eq:block}
\Z^d[B] = \dsum_{j=1}^r \bdd_j[H_j] = (\dsum_{j=1}^r \bdd_j)[\oplus_{j=1}^r H_j] = \bdd[B].
\end{equation}
 In other words, the  digit system $(\oplus_{j=1}^{r} H_j, \dsum_{j=1}^{r} \bdd_j) = (B, \bdd)$, and it has the finiteness property in $\Z^d[B]$. The dynamical system $\Phi = \oplus_{j=1}^r \Phi_j$ then performs the division with remainder by $B$ in $\Z^d[B]$.

\medskip

\noindent \emph{Case b)} is more subtle. By Corollary \ref{col:RestrExt} and Lemma \ref{lemHypercomp}, for each (expanding or non-expanding) $H_j$ of order $k_j \geq 2$ and dimension $d_j \times d_j$, there exists an integer digit set $\bdd_{j} \subset \Z^{d_j}$ containing $\bo_{d_j}$, and a restricted dynamical system $\Phi_{\mathbf{r},H_j}: \Z^{d_j} \to \Z^{d_j}$, that has attractor $\{\bo_{d_j}\}$ in $\Z^{d_j}$. Then, one constructs the extended division function $\Psi_j$ in $\Z^{d_j}[H_j]$, as in \eqref{efExtDiv}, that performs  division with remainder by $H_j$ in $\Z^{d_j}[H_j]$ with respect to $\bdd_{j}$. Then $\Psi_j$ has attractor $\{\bo_{d_j}\}$ in $\Z^{d_j}[H_j]$, and $\bdd_{j}[H_j]=\Z^{d_j}[H_j]$ by Lemma \ref{lem:attrExt}. From that point on, one takes the direct sum \eqref{eq:block} over every $\bdd_{j}[H_j]$ as in case (a), only using the extended division $\Psi_j$ in $\Z^{d_j}[H_j]$ in the place of $\Phi_j$, whenever $k_j \geq 2$ and $H_j$ is not expanding. 

Thus, in both cases (a) or (b), one arrives to the situation $\Z^d[B]=\bdd[B]$, for some finite digit set $\bdd \subset \Z^d[A]$, which can be also be selected from $\Z^d$, if necessary. The last step is the same as in the proof of \cite{JanThu}*{Theorem~2, p.~357}, after conjugating  with $T \in \Z^{d \times d}$, and using $TBT^{-1}=A$, one obtains $(T\Z^d)[A] = \bdd'[A]$, with $\bdd'=T\bdd$. As $T\Z^d \subset \Z^d$ is a sublattice, a set of residue class representatives $\rr$ of $\Z^d/T\Z^d$ is finite, and $\Z^d = TZ^d+\rr$. This gives
\[
\Z^d[A] = (T\Z^d + \rr)[A] = (T\Z^d)[A] + \rr[A] = \bdd'[A] + \rr[A] = \bdd''[A],
\] for the finite digit set $\bdd'' := \bdd' + \rr$.

For the converse part of the Theorem \ref{thm:Main}, we proceed as in \cite{Vin1}, \cite{Gi2}, or \cite{JanThu}. Assume that $(A, \bdd)$ has the finiteness property. Pick arbitrary $\bz = \bd_0 + A\bd_1 + \dots + A^k\bd_k \in \bdd[A]$ and let $\bv \in \C^d$ satisfy $\bv^tA=\la \bv^t$. If $\abs{\la} < 1$, then 
$\bv^t\bz = \bv^t\bd_0 +  \la\bv^t\bd_1 + \dots \la^k\bv^t\bd_k$,
As $\bdd$ is finite, pick $C > 0$, such that $\abs{\bv^t\bd} < C$, for every $\bd \in \bdd$. Then, \[
(\Re{\bv})^t\bz = \Re{(\bv^t\bz)} \leq \abs{\bv^t\bz} < C(1+\abs{\la}+ \dots + \abs{\la}^k) < C/(1-\abs{\la}).
\]
In particular, this means that all integer vectors $\Z^d \subset \bdd[A]$ lie in a half-plane, as $\Z^d[A]=\bdd[A]$, but that is impossible. Therefore, every eigenvalue $\la$ of $A$ must be $\abs{\la} \geq 1$.
\end{proof}

We now derive Theorem \ref{thm:ATZ} from Theorem \ref{thm:Main}.

\begin{proof}[Proof of Theorem \ref{thm:ATZ}]
Assume that no algebraic conjugate of $\al$ over $\Q$ is of absolute value $<1$.
Let $A=C(f) \in \Q^{d \times d}$  be the companion matrix of the minimal polynomial $f \in \Z[x]$ of $\al$. As the eigenvalues of $A$ are all $\geq 1$ in absolute value then, according to Theorem~\ref{thm:Main}, one can find a finite set $\bdd \subset \Z^d$, such that
\begin{equation}\label{eq:incl1}
\Z[A]\boe_1 \subseteq \Z^d[A] = \bdd[A],
\end{equation}
where $\Z[A]$ is the set of matrix polynomials in $A$ with integer coefficients and the standard basis $\boe_j \in \Z^d, 1 \leq j \leq d$ satisfy $A\boe_j=\boe_{j+1}$, for $1 \leq j \leq d-1$. Let $D \subset \Z^d$ denote the set of the integer coordinate vectors of the digits $\bd_i=\sum_{j=1}^dd_{ij}\boe_j \in \bdd$ w.r.t.\ the standard basis. Define the finite set $F :=\underbrace{D + D + \dots +D}_{d-\text{times}}$. Write
$\bd_i
=\sum_{j=1}^dd_{ij}A^{j-1} \boe_1$.
For arbitrary $\bz \in \bdd[A]$,
\[
\bz = \sum_{j=0}^{k} A^i \bd_i = \sum_{i=0}^{k}A^i \left(\sum_{j=1}^d d_{ij} A^{j-1} \boe_1 \right) = \sum_{i=1}^{k+d} A^{i-1} \left(\sum_{j=1}^{\min\{d, i\}} d_{i-j,j} \boe_1 \right)= \sum_{j=1}^{k+d} f_i A^{i-1} \boe_1,
\]
where $f_i := \sum_{j=1}^{\min\{d, i\}} d_{i-j,j} \in F$. Hence,
$\bdd[A] \subset F[A]\boe_1$. In view of \eqref{eq:incl1},
$\Z[A]\boe_1 \subseteq F[A]\boe_1 \subseteq \Z[A]\boe_1$, since $F \subset \Z$. Therefore, $\Z[A]\boe_1 = F[A]\boe_1$. Now pick the right eigenvector $\bv \in \C^d$ of $A$, such that $\bv^t \boe_1 \ne 0$ (such $\bv$ exists, because the companion matrix $A$ of the minimal polynomial of $\al \ne 0$ is diagonalizable over $\C$ and invertible). Then $\bv^t A=\al'\bv$, where $\al'$ is algebraically conjugate to $\al$ over $\Q$. Since $\bv^t\Z[A]\boe_1 = \Z[\al']\bv^t\boe_1$,  $\bv^t F[A]\boe_1 = F[\al']\bv^t\boe_1$, $\bv^t\boe_1 \ne 0$, we obtain $\Z[\al']=F[\al']$. By applying the Galois automorphism $\al' \mapsto \al$, we find that $\Z[\al]=F[\al]$.

Conversely, assume that $\Z[\al]=F[\al]$ for some finite set $F \in \C$. Then, $F$ is a subset of $\Z[\al]$ itself. Suppose that $\al$ has an algebraic conjugate over $\Q$ of absolute value $<1$. We proceed in the same way as in the lattice digit system case. By taking Galois automorphism $\al \mapsto \al'$, we see that $\Z[\al']=F'[\al']$, for some finite set $F' \subset \Z[\al']$. Therefore, we can assume that $\al$ itself is of the absolute value $<1$. Then, an arbitrary element of $\be \in F[\al]$, $\be = f_0 + f_1 \al + \dots + f_n \al^n$, $f_j \in F$ is of absolute value at most
\[
\abs{\be} \leq C(1 + \abs{\al} + \dots + \abs{\al^n}) < C/(1-\abs{\al}), 
\] where $C = \max_{f \in F}\abs{f}$. This shows that $F[\al]$ is a bounded subset of $\C$, which is impossible, because it contains $\Z$, if $F[\al]=\Z[\al]$.
\end{proof}

\section{Maximal integral sublattice of a rational lattice}\label{sec:theorySNF}

For practical computations in Section~\ref{sec:num}, the auxiliary lattice of the form $\ll=\Z^d \cap A\Z^d$ (that is, `the integral part' of $A\Z^d$), where $A$ is an invertible rational matrix, and the group of its coset representatives $\Z^d/\ll)$ play very important roles. In particular, $\Z^d/\ll$ can be viewed as 'the approximation of the first order' to the group $A\Z^d[A]/\Z^d[A]$, hypothetically the smallest necessary subset of digits needed to represent elements of $\Z^d[A]$. Since it is desirable to control the size of a digit set we will need a formula for the basis matrix and the coset representatives of this auxiliary lattice $\ll$ in $\Z^d$.

\subsection{Review of lattices} We recall basic facts on lattices. An additive subgroup $\ll \subset \R^d$ is called \emph{a lattice}, if $\ll$ is a uniformly discrete and relatively dense subset of $\R^d$. In this case, $\ll$ is a free $\Z$- module of rank $d$ and there exists a set 
$\{\bv_1, \bv_2,\dots, \bv_d\}$ of $d$ linearly independent (over $\R$) vectors in $\R^d$  that generate $\ll$ over $\Z$, so that
$\ll = \Z \bv_1+\Z \bv_2 + \dots + \Z \bv_d$ (\cite{Neu}{Proposition~4.2 of Section~4}) . This can be also written as $\ll = L \Z^d$, where $L \in \R^{d \times d}$ denotes the basis matrix with columns $\bv_1$, $\dots$, $\bv_d$. In particular, if $L \in \Q^{d \times d}$ or $L \in \Z^{d \times d}$, then $\ll$ is called \emph{rational}, or \emph{integral} lattice, respectively. The matrix $L$ is not unique and depends on the choice of the basis for $\ll$. Two lattices $\ll = L \Z^d$ and $\mm = M \Z^d$, $L, M \in \R^d$, are nested $\ll \subset \mm$ if and only if $L$ is left divisible by $M$ in $\Z^{d \times d}$, {\it i.e.}, if  there exist an integer matrix $Q \in \Z^{d \times d}$, such that $L=MQ$. In particular, $\ll=\mm$ if and only if $Q$ is \emph{a unimodular} matrix, that is, $\det(Q) = \pm 1$; in that case $M=LQ^{-1}$. In particular, for any unimodular $Q \in \Z^{d\times d}$, one has $Q\Z^d=\Z^d$. 



From the theory of the Smith Normal Form (SNF) \cites{New72, New97}, for an arbitrary invertible matrix $B \in\Z^{d\times d}$ there exist unimodular matrices $P, Q \in \Z^{d \times d}$, such that $B=PDQ^{-1}$, where $D=\mathrm{SNF}(B)$ is the unique diagonal matrix,
\begin{equation}\label{snf}
D = \begin{pmatrix}
n_1 & 0 & \dots & 0\\
0 & n_2 & \dots & 0\\
\vdots & \vdots & \ddots & \vdots\\
0 & 0 & \dots & n_d\\
\end{pmatrix},
\end{equation}
whose entries $n_1$, $\dots$, $n_d$, called the {\em elementary divisors}, are positive integers satisfying the divisibility properties $n_1 \mid n_2$, $ n_2\mid n_3$,\,$\dotsc$,\,$ n_{d-1} \mid n_d$. The matrices $P$, $Q$, and $D$ can be found by reducing the rows and columns of $B$ through the integer division with remainder (which corresponds to invertible base transformations in $\Z^d$ and $\ll=B\Z^d$). A few other facts on SNF that we will need further are summarized in Proposition \ref{prop:resGroup}.

\begin{proposition}\label{prop:resGroup} Let $\ll = L\Z^d$, $L \in \Z^{d \times d}$ be a lattice that has a {\rm SNF} factorization $L=PDQ^{-1}$. Then
$\Z^d = L\Z^d + \rr$, where the $\rr$ is the full set of coset representatives 
\begin{equation}\label{eq:resGroup}
\mathcal{S} := \left\{ \sum_{i=1}^d j_i P\boe_i, 0 \leq j_i \leq n_i-1\right\}.
\end{equation}
Here, $\boe_i$, $i=1,\dots, d$ is the standard basis of $\Z^d$ and $n_i$ are the elementary divisors from $D=\mathrm{SNF}(L)$. Consequently, $[\Z^d : \ll] = \abs{\rr}= n_1 \cdots n_d = \abs{\det(D)}=\abs{\det(L)}$.
\end{proposition}
\begin{proof}[Proof of Proposition \ref{prop:resGroup}]
As $D$ is diagonal, $\Z^d/D\Z^d \simeq \Z_{n_1} \oplus \Z_{n_2} \oplus \dots \oplus \Z_{n_d}$ with the set of representatives $\tt := \left\{ \sum_{i=1}^d j_i\boe_i, 0 \leq j \leq n_i-1 \right\}$. Hence, $\Z^d = D\Z^d + \tt$. By the unimodularity of $P$ and $Q$, $\Z^d = P\Z^d = PD\Z^d + P\tt=PDQ^{-1}\Z^d+P\tt$. Setting $\mathcal{S} := P\tt$ we obtain $\Z^d = PDQ^{-1}\Z^d + P\tt=L\Z^d+\mathcal{S}$. Notice that $P\bx \in \Z^d$ belongs to $L\Z^d = PDQ^{-1}\Z^d = PD\Z^d$ if and only if $\bx \in D\Z^d$, therefore $\Z^d/D\Z^d$ and $\Z^d/L\Z^d$ are isomorphic. The last statement follows.
\end{proof}

\subsection{A result on rational lattices}

We turn our attention to rational lattices $\ll = A\Z^d$, where $A \in \Q^{d \times d}$ is invertible. Let $q \in \N$ be a common denominator to the fractions that appear in the entries of $A$ (not necessary the smallest one). Then $A = q^{-1}B$, where $B \in \Z^{d \times d}$. We are interested in the largest integral sub-lattice $\Z^d \cap \ll = \Z^d \cap (q^{-1}B)\Z^d$.

\begin{theorem}\label{thmBasis}
Let $B\in \Z^{d \times d}$ have {\rm SNF} factorization $B=PDQ^{-1}$, with elementary divisors $n_1$, $\dots$, $n_d$, and let $q \in \N$. Then the matrices $L=PN$ and $L'=LQ^{-1}= PNQ^{-1}$, with
\[
N=\begin{pmatrix}
n_1/\gcd(q, n_1) & 0 & \dots & 0\\
0 & n_2/\gcd(q, n_2) & \dots & 0\\
\vdots & \vdots & \ddots & \vdots\\
0 & 0 & \dots & n_d/\gcd(q, n_d)\\
\end{pmatrix}
\]
where $n_1$, $n_2$, $\dots$, $n_d$ stands for the elementary divisors of $D=SNF(B)$, are two possible basis matrices for the maximal integral sub-lattice $\Z^d \cap \ll$ of the rational lattice $\ll := q^{-1} B\Z^d \subset \Q^d$. The set 
\[
\mathcal{S}=\left\{ \sum_{i=1}^d j_iP\boe_i, 0 \leq j_i < n_i/\gcd(n_i, q)\right\}
\]
consists of all distinct coset representatives of $\Z^d/(\Z^d \cap \ll)$.
\end{theorem}

\begin{proof}[Proof of Theorem \ref{thmBasis}]
Consider the scaled lattice $\mm := q\cdot \left(\Z^d \cap \ll\right) = q\Z^d \cap B\Z^d$. Then
\[
\mm' := P^{-1}\mm = (qP^{-1})\Z^d \cap (P^{-1}B)\Z^d = (qP^{-1}\Z^d) \cap (P^{-1}B)\Z^d=q\Z^d \cap (P^{-1}B)\Z^d,
\]
since $P^{-1}\Z^d=\Z^d$. Also, $(P^{-1}B)\Z^d = P^{-1}BQ\Z^d=D\Z^d$. Therefore, $\mm' = q\Z^d \cap D\Z^d$. As $q\Z^d$ consists of vectors whose entries are divisible by $q$ and $D\Z^d$ consists of vectors whose $i$-th entry is divisible by $n_i$ by \eqref{snf}, $1 \leq i \leq d$, $q\Z^d \cap D\Z^d$ consists of vectors whose $i$-th entry is divisible by $\text{lcm}(q, n_i)$. In other words $\mm'= C \Z^d$, where $C$ is
\[
C = \begin{pmatrix}
\text{lcm}(q, n_1) & 0 & \dots & 0\\
0 & \text{lcm}(q, n_2) & \dots & 0\\
\vdots & \vdots & \ddots & \vdots\\
0 & 0 & \dots & \text{lcm}(q, n_d)\\
\end{pmatrix}.
\]
Scaling back, $\Z^d \cap \ll = q^{-1}\mm =P(q^{-1}C)\Z^d = PN\Z^d = PNQ^{-1}\Z^d$. Therefore, one can take $L:=PN$ or$L'=PNQ^{-1}$ to be the basis matrix of $\Z^d \cap \ll$. Coset representatives are obtained from \eqref{eq:resGroup} in Proposition \ref{prop:resGroup}. 
\end{proof}

In order to state the next result, recall that \emph{the content} \cite{And} of the polynomial $f(x) = a_dx^d+\dots+a_1x+a_0 \in \Z[x]$ is  defined by $c(f) := \text{gcd}(\abs{a_0}, \abs{a_1}, \dots, \abs{a_n})$. By the Gauss Lemma \cite{And}, for any two polynomials $f, g \in \Z[x]$, one has $\mc(f\cdot g)=\mc(f)\cdot \mc(g)$.

\begin{theorem}\label{thmIndex}
Let $B \in \Z^{d \times d}$ be invertible and $q \in \N$. The index of the maximal integral sublattice $\Z^d \cap \ll$ of $\ll = q^{-1}B\Z^d$in $\Z^d$ is 
\[
[\Z^d : \left(\Z^d \cap \ll\right)] = \frac{\abs{\det(B)}}{\mc\left(\varphi_D(qx)\right)},
\]
where and $\varphi_D(x)$ is the characteristic polynomial of $D=\mathrm{SNF}(B)$.
\end{theorem}
\begin{proof}[Proof of Theorem \ref{thmIndex}]
By Theorem \ref{thmBasis},
\[
[\Z^d : (\Z^d \cap \ll)]  = \det{N} =\prod_{j=1}^d \frac{n_j}{\text{gcd}(q, n_j)}= \frac{\abs{\det(B)}}{\prod_{j=1}^d\text{gcd}(q, n_j)}.
\]
The content of $n_j - qx$ is exactly $\text{gcd}(q, n_j)$. Since $\varphi_D(x) = \prod_{j=1}^d(n_j-x)$, one has
$\prod_{j=1}^d\text{gcd}(q, n_j) = \prod_{j=1}^d \mc(n_j-qx)=\mc\left(\varphi_D(qx)\right)$.
\end{proof}



\section{Numerical examples}
\label{sec:num}

We illustrate the main features of the theory of digit systems developed so far with examples.

\subsection{Rotation digit system associated to smallest Pythagorean triple.}\label{subsec:rot345}

Consider the $2 \times 2$ matrix
\[
A = \begin{pmatrix*}[r]3/5&-4/5\\4/5&3/5\end{pmatrix*},
\] that corresponds to the Pythagorean triple $3^2+4^2=5^2$. This is a rotation by the angle $\theta = 53.1301023542\dots^{\circ}$. We shall construct a small digit set $\bdd$, such that $(A, \bdd)$ in $\Z^2[A]$ has the finiteness property.

First, we determine the basis for the auxiliary lattice $\ll :=\Z^2 \cap A\Z^2$. Let
\[
B := 5A=\begin{pmatrix*}[r]
				3 &-4\\
                     4&3
			\end{pmatrix*}.
\]
Then, $\ll = \Z^2 \cap A\Z^2 = \frac{1}{5}\left(5\Z^2 \cap B\Z^2\right)$.
The SNF decomposition yields $B=PDQ^{-1}$ with
\[
D = \begin{pmatrix}
			1 & 0\\
            0 & 25
			\end{pmatrix},\quad
P = \begin{pmatrix}
		7 & 1 \\
		1 & 0
	\end{pmatrix},\quad
Q=\begin{pmatrix*}[r]
		1 & 3 \\
		-1 & -4
	\end{pmatrix*}.
\] with $\det{P}=\det{Q}=-1$. According to Theorem \ref{thmBasis}, $\ll = L\Z^d$ with the basis matrix
\[
L=PN = \begin{pmatrix}
		7 & 1 \\
		1 & 0
	\end{pmatrix} \cdot
	\begin{pmatrix}
		1/\text{gcd}(5, 1) &0\\
     	0&25/\text{gcd}(5, 25)
	\end{pmatrix}=
	\begin{pmatrix}
		7 & 1 \\
		1 & 0
	\end{pmatrix} \cdot
	\begin{pmatrix}
		1 &0\\
        0 &5
	\end{pmatrix}=
	\begin{pmatrix}
		7 & 5\\
        1 & 0
	\end{pmatrix}.
\]
One possible set of coset representatives of $\Z^2$ modulo $\ll$ (equivalent to the one given in Theorem \ref{thmBasis}) is
\[
\mathcal{S}=\Z^2/\ll = \left\{ 0P\boe_1+0P\boe_2, 0P\boe_1+1P\boe_2, 0P\boe_1+2P\boe_2, 0P\boe_1+3P\boe_2, 0P\boe_1+4P\boe_2 \right\} =
\]
\[
= \left\{\vectwo{0}{0}, \vectwo{1}{0}, \vectwo{2}{0}, \vectwo{3}{0}, \vectwo{4}{0}\right\}=
\]
After translating $\mathcal{S}$ by $-2\boe_1$, we obtain an equivalent set of representatives with reduced coordinates:
\[
\rr := \mathcal{S}-2\boe_1 = U^{-1}\sc = \left\{\vectwo{-2}{0}, \vectwo{-1}{0}, \vectwo{0}{0}, \vectwo{1}{0}, \vectwo{2}{0}\right\}=
\]
\[
=\{-2\boe_1, -\boe_1, \bo, \boe_1, 2\boe_1\}.
\]

Next, we construct a \emph{pre-periodic} digit set $\dd_{\mathrm{r}}'$ that has a good convex enclosure with respect to $\rr$ as depicted as points in Figure \ref{fig_enclosures}.

\begin{figure}
	\includegraphics{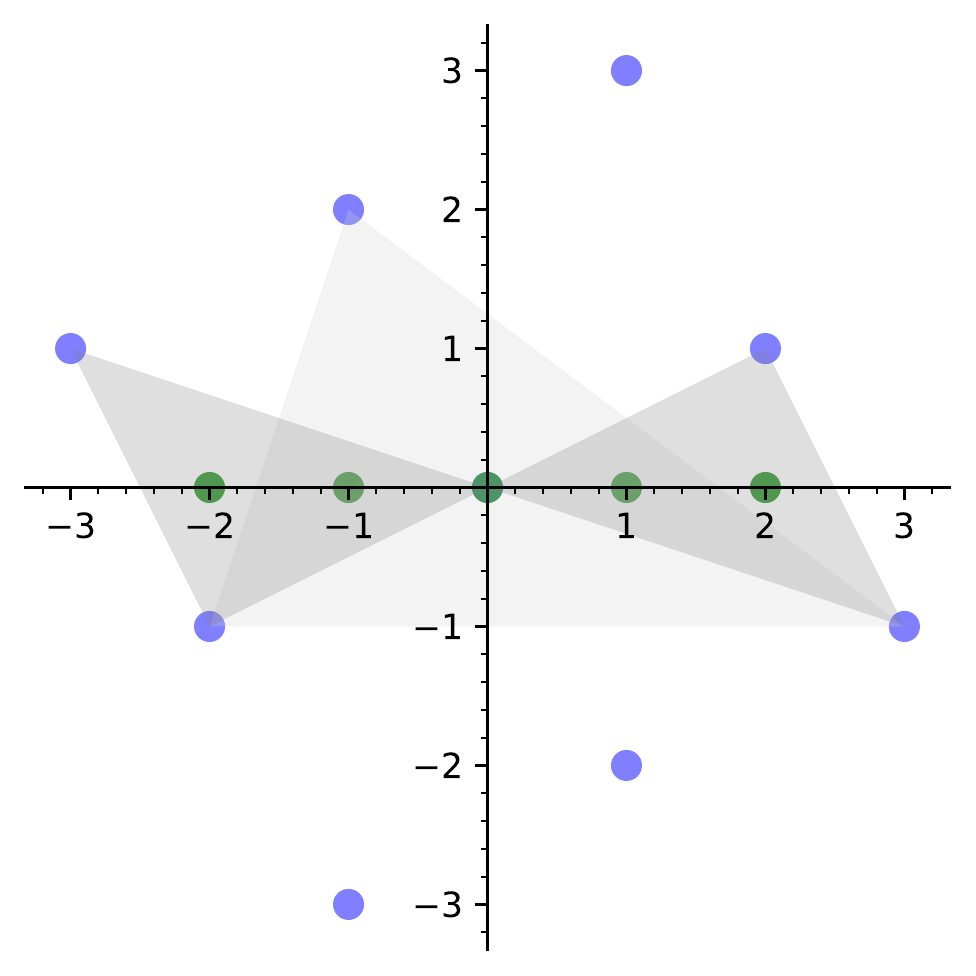} 
	\caption{Points from $\ll = \Z^2 \cap A\Z^2$ (blue) and from $\rr=\Z^2/\ll$ (green). Two small dark grey triangles, namely, $\tt_1$ and $-\tt_1$  (grey) and one big light grey triangle $\tt_2$ with vertices in $\ll$ enclose elements of $\rr$.\label{fig_enclosures}}
\end{figure}

Notice that the residues $\boe_1$ and $2\boe_1$ are the inner points of a triangle with vertices
\[
\tt_1 = \left\{\vectwo{0}{0}, \vectwo{2}{1}, \vectwo{3}{-1}\right\} \subset \ll.
\] Likewise, $-\boe_1, -2\boe_1$ are the inner points of the triangle with vertices in $-\tt_1$. The last residue $\bo_2$ is an inner point of the triangle with vertices
\[
\tt_2 = \left\{\vectwo{-2}{-1}, \vectwo{-1}{2}, \vectwo{3}{-1}\right\} \subset \ll.
\]

\begin{figure}
	\begin{subfigure}{0.4\textwidth}
		\includegraphics[scale=0.65]{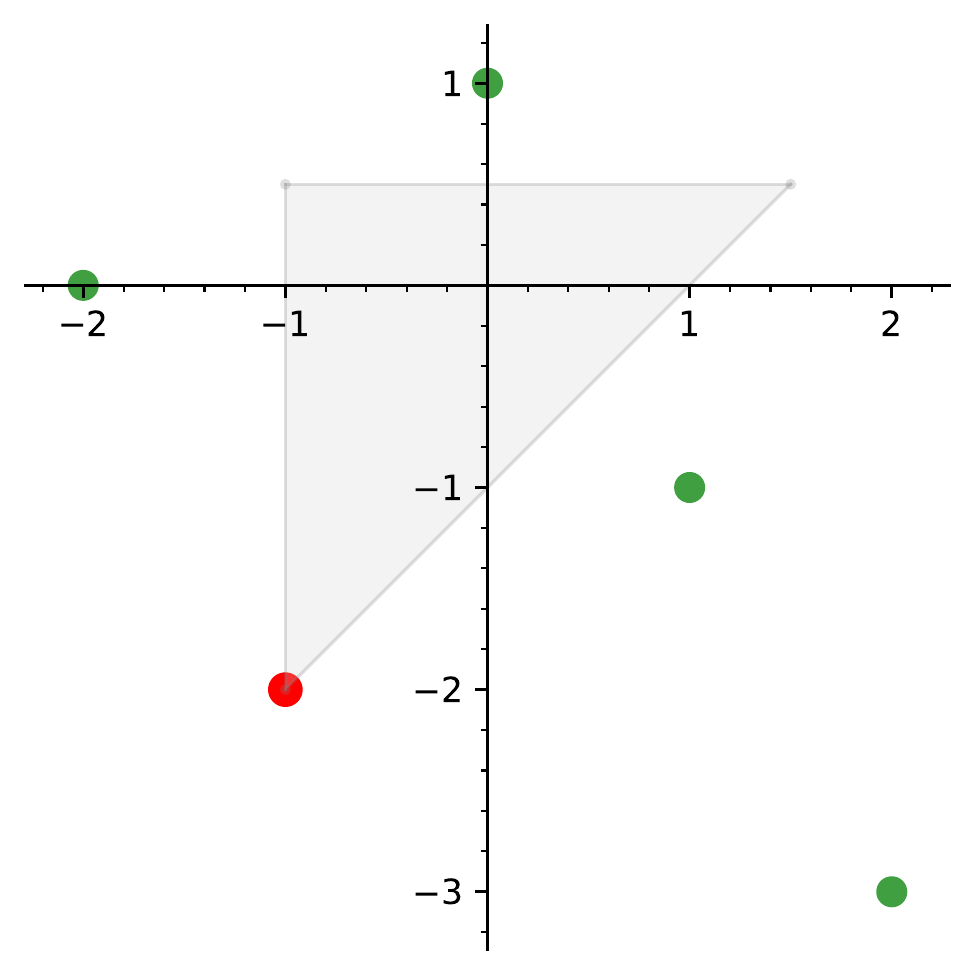} 
		\caption{$\bx=(-2, 0)^T$}
	\end{subfigure}
	\hfill
	\begin{subfigure}{0.4\textwidth}
		\includegraphics[scale=0.65]{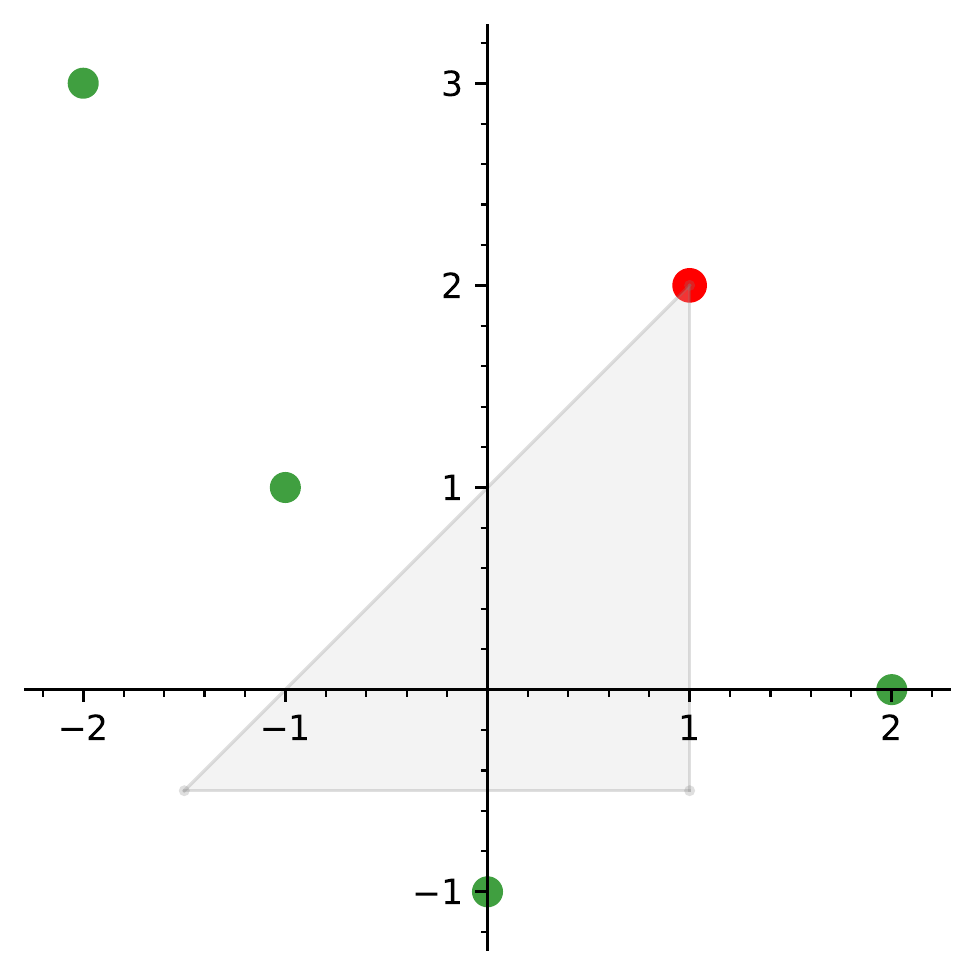} 
		\caption{$\bx=(2, 0)^T$}
	\end{subfigure}
	\begin{subfigure}{0.4\textwidth}
		\includegraphics[scale=0.65]{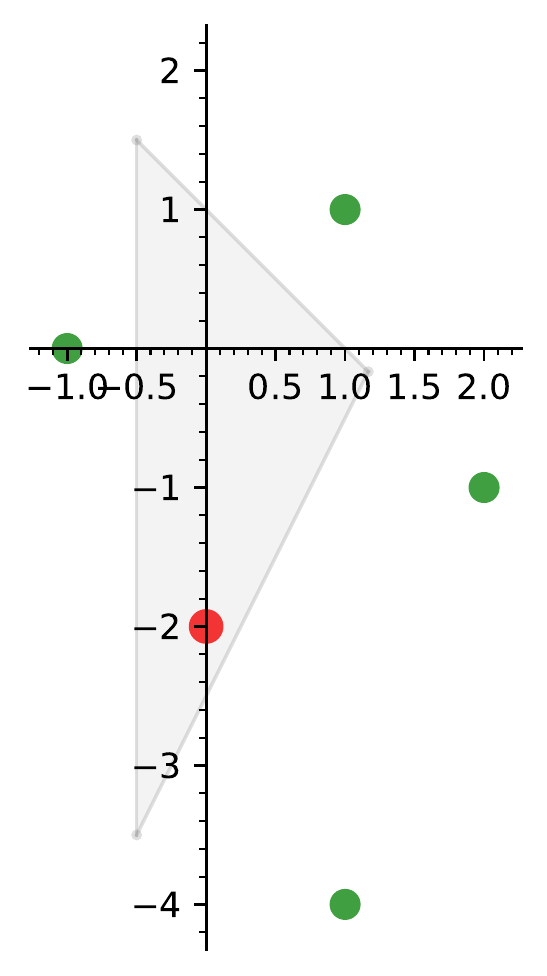} 
		\caption{$\bx=(-1, 0)^T$}
	\end{subfigure}
	\hfill
	\begin{subfigure}{0.4\textwidth}
		\includegraphics[scale=0.65]{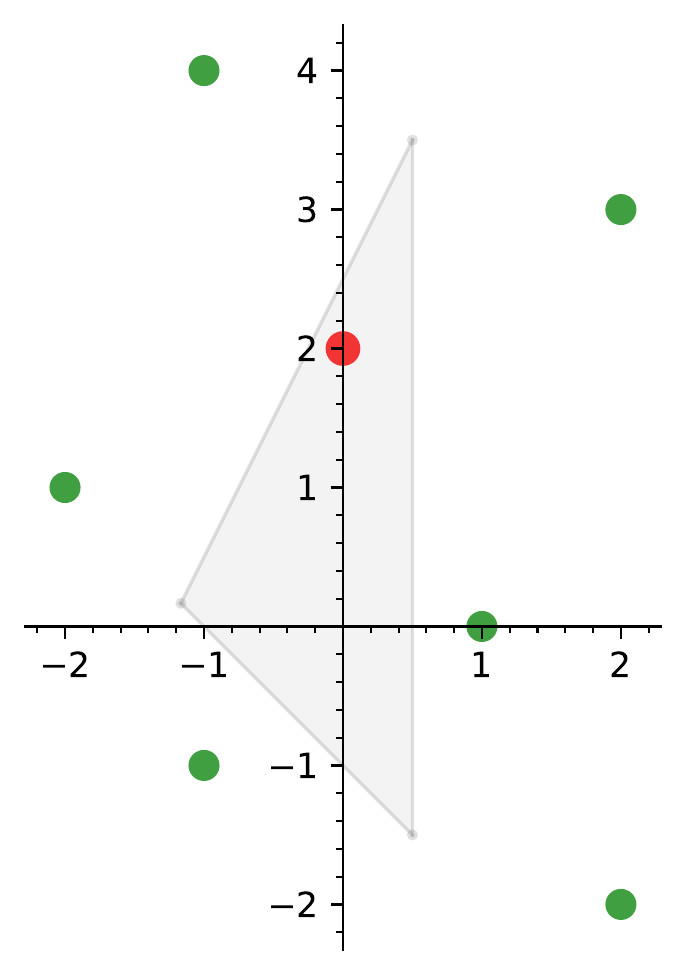} 
		\caption{$\bx=(1, 0)^T$}
	\end{subfigure}
	\begin{subfigure}{\textwidth}
		\centering
		\includegraphics[scale=0.65]{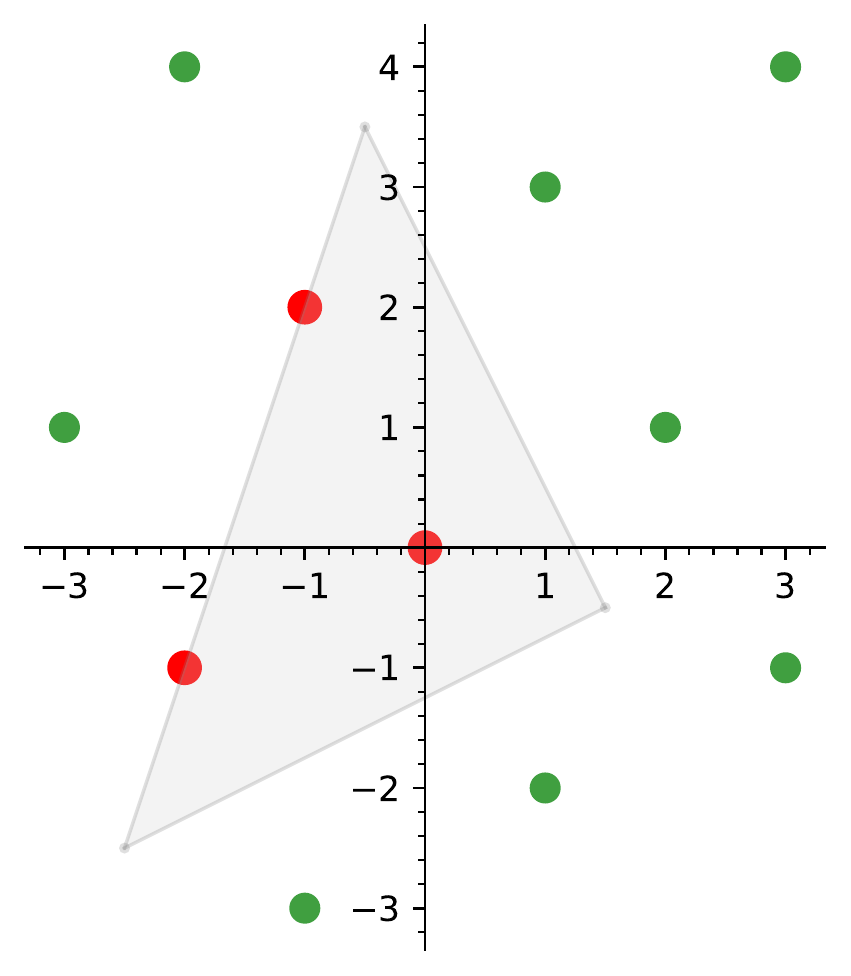} 
		\caption{$\bx=(0, 0)^T$}
	\end{subfigure}
	\caption{Repeller points $\rp{\Phi_{\mathrm{r}}}{\bx}$  (in red) inside the regions of repelling action of $\Phi_{\mathrm{r}}$ (in grey), for each residue class $\bx \in \Z^2/\ll$. Green colored are the remaining points from the same residue class $\bx +\ll$.\label{fig_feasible}}
\end{figure}
Define the sets
\[
\bdd(\bo) := \bo - \tt_2 =  \left\{\vectwo{2}{1}, \vectwo{1}{-2}, \vectwo{-3}{1}\right\},
\]
\[
\bdd(\boe_1) := \boe_1 - \tt_1= \left\{\vectwo{-1}{-1}, \vectwo{1}{0}, \vectwo{-2}{1}\right\},
\]
\[
\bdd(-\boe_1) := -\boe_1 + \tt_1 =\left\{\vectwo{-1}{0}, \vectwo{2}{-1}, \vectwo{1}{1}\right\}.
\]
\[
\bdd(2\boe_1) := 2\boe_1 - \tt_1=\left\{\vectwo{0}{-1}, \vectwo{2}{0}, \vectwo{-1}{1}\right\},
\]
\[
\bdd(-2\boe_1) := -2\boe_1+\tt_1 =\tt\left\{\vectwo{-2}{0}, \vectwo{1}{-1}, \vectwo{0}{1}\right\}.
\]
As $\bo_2$ is an inner point of these sets, their union
\[
\bdd' = \bigcup_{\bx \in \rr} \bdd(\bx)
\]  has good enclosure with respect to $\rr$. By Theorem \ref{thmPeriodic} and Remark \ref{rem:rotRestr}, the division mapping $\Phi_{\mathrm{r}}(\bx)=A^{-1}(\bx-\bd_r(\bx))$ with the digit $\bd_{\mathrm{r}}(\bx) \in \bdd'$, restricted to $\Z^2$, has a finite attractor $\aa_{\Phi_{\mathrm{r}}}$ in $\Z^2$. To compute the set $\aa_{\Phi_{\mathrm{r}}}$, first notice that our matrix $A$ is orthogonal (it is a rotation matrix): by taking the identity matrix $T=\id_2$ in the definition of $A$--invariant norm form on p.11, one obtains that the usual Euclidean norm $\norm{\dots}_{2}$ is $A$ and $A^{-1}$--invariant. To determine the attractor set $\aa_{\Phi_{\mathrm{r}}}$, one needs to find the repeller points $\bx \in \Z^2$, such that $\norm{\Phi_{\mathrm{r}}(\bx)}_{A^{-1}} \geq \norm{\bx}_{A^{-1}}$ and then determine their orbits. In our case, we need to compute all the integral points in the feasible regions  $\bigcup_{\bx \in \rr}\vv(\bx)$, such that
\[
 \vv(\bx) := \{\by \in \R^n: \norm{\bx - \bd} \geq \norm{\bx}, \forall \bd \in \bdd'(\bx)\}.
\]
By Corollary \ref{corMinv} the feasible regions $\vv(\bx)$ are compact since $\bdd'$ has good enclosure with respect to $\rr$. For each $\bx \in \rr$, we write down and solve the linear program that corresponds to $\vv(\bx)$, as in the proof of Corollary \ref{corMinv}.  We use \texttt{SAGE} \cite{sagemath} \texttt{Polyhedron} and \texttt{Mixed Integer Linear Programming (MILP)} modules in a combination with \texttt{PPL} solver to explicitly solve for the feasible regions and enumerate the sets of contained integer points
\[
\rp{\Phi_{\mathrm{r}}}{\bx} :=\vv(\bx) \cap (\bx + \ll),
\] refer to Figure~\ref{fig_feasible}. Our calculations yield:
\[
\rp{\Phi_{\mathrm{r}}}{-2\boe_1} = \left\{ \vectwo{-1}{-2}\right\},\qquad \rp{\Phi_{\mathrm{r}}}{2\boe_1} = \left\{ \vectwo{1}{2}\right\},
\]
\[
\rp{\Phi_{\mathrm{r}}}{-\boe_1} = \left\{ \vectwo{0}{-2}\right\},\qquad \rp{\Phi_{\mathrm{r}}}{\boe_1} = \left\{ \vectwo{0}{2}\right\},
\] and
\[
\rp{\Phi_{\mathrm{r}}}{\bo} = \left\{ \vectwo{-1}{2},   \vectwo{0}{0}, \vectwo{-2}{-1}\right\}. 
\]
The last step is to optimize the digit set by picking up the minimal possible number of representatives from the periods in orbits of points in $\rp{\Phi_{\mathrm{r}}}{\bx}$. To find this minimal set, we compute the images
\[
\Phi_{\mathrm{r}}\vectwo{0}{0} = \Phi_{\mathrm{r}}\vectwo{0}{2} = \vectwo{1}{2}, \qquad \Phi_{\mathrm{r}}\vectwo{0}{-2} = \Phi_{\mathrm{r}}\vectwo{-2}{-1} = \vectwo{-1}{-2},
\]
\[
\quad \Phi_{\mathrm{r}}\vectwo{-1}{2} = \vectwo{2}{-1}, \quad \Phi_{\mathrm{r}}\vectwo{2}{-1} = \vectwo{0}{0}
\]
and,
\[
\Phi_{\mathrm{r}}\vectwo{1}{2} = \vectwo{1}{2}, \quad \Phi_{\mathrm{r}}\vectwo{-1}{-2} = \vectwo{-1}{-2}.
\]
Therefore, any orbit of a point $\bx \in \Z^2$ eventually hits the set
\[
\pp = \left\{\vectwo{1}{2}, \vectwo{-1}{-2}\right\}.
\]
For the final digit set, we take
\[
\bdd = \bdd' \cup \pp.
\]
It consists of $17$ digits, all depicted in Figure~\ref{fig_digits}. Any $\Z^2$ vector has a finite radix expansion in base $A$ with digits $\bdd$. For instance, the radix expansion of $(6, -7)^t \in \Z^2$ is
\[
\vectwo{6}{-7} = \vectwo{1}{-2} + A\vectwo{1}{-1} + A^2\vectwo{-2}{0} + A^3\vectwo{-1}{1} + A^4\vectwo{0}{1} +A^5\vectwo{1}{2} \in \bdd[A].
\]

\begin{figure}
	\includegraphics[scale=0.75]{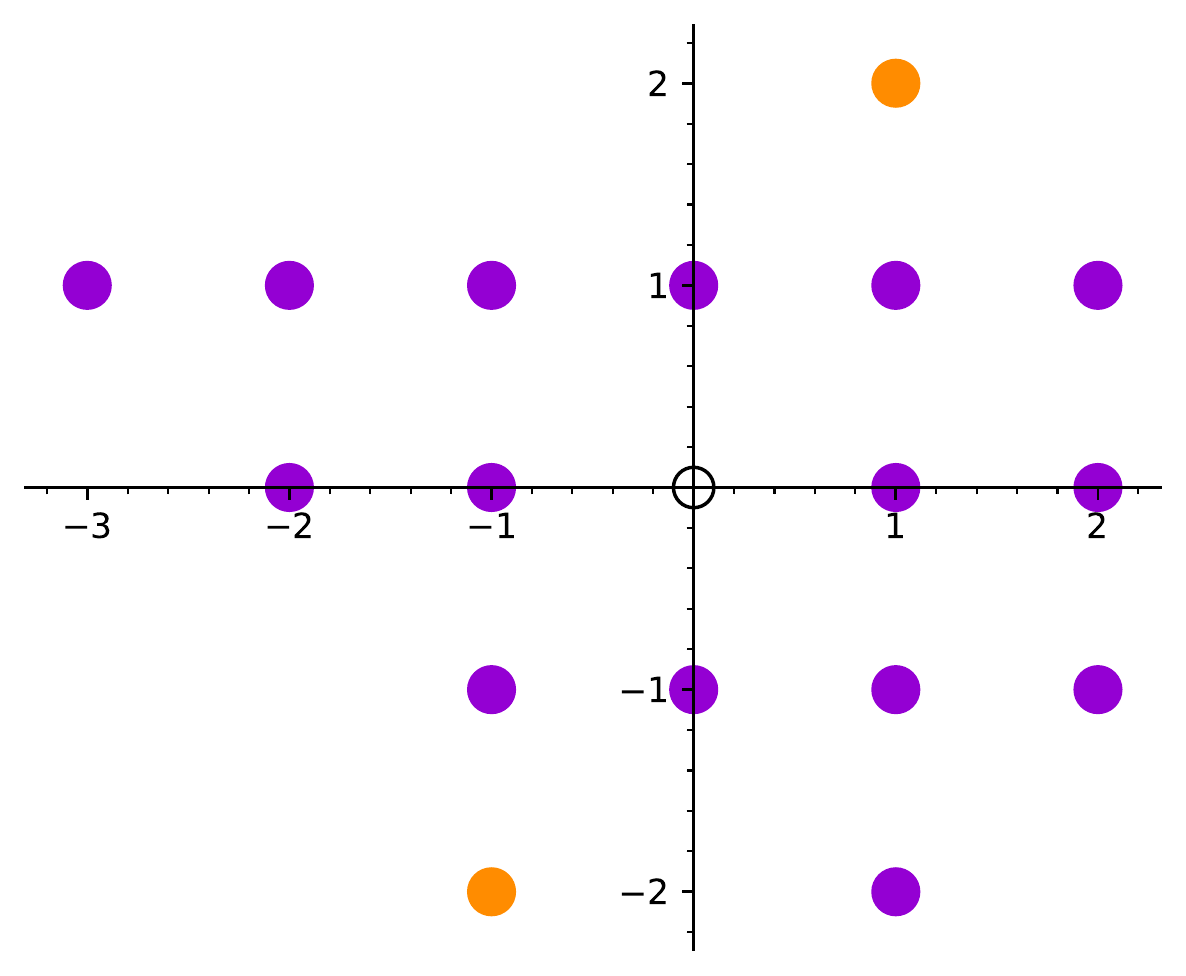} 
	\caption{The digit set $\bdd$: the initial digits  $\bdd'$ (violet) and additional points from $\aa_{\Phi_{\mathrm{r}}}$(orange) that represent periods; notice that $\bo_2 \not\in \bdd$.\label{fig_digits}}
\end{figure}

It is remarkable that non-trivial rotational digit systems that possess the finiteness property in principle do not require zero digit $\bo_d \in \bdd$. In the present example, the zero vector $\bo_2$ is represented in $\bdd[A]$ as
\[
\vectwo{0}{0} = \vectwo{2}{1} + A\vectwo{-2}{1}.
\]
In some situations (like taking the twisted sums of the rotational digit systems), the  artificial inclusion of $\bo_d$ in $\bdd$ might be necessary, see Section \ref{subsec:example_twisted}.

The finite digital expansions of vectors from lattice from $\Z^2$ in the digit system $(A, \bdd)$ yield the finite expansions of vectors from module $\Z^2[A]$. Let us illustrate the expansion process for the vector
\[
v=\vectwo{-156/25}{-8/25}= \vectwo{-3}{0}+A\cdot\vectwo{0}{2}+A^2\vectwo{-1}{2} \in \Z^2[A].
\]
By applying the mappings $\bd_{\mathrm{r}}$, $\Phi_{\mathrm{r}}$, one expresses $\vectwo{-3}{0}=\vectwo{-1}{1}+A\vectwo{-2}{1}$ and then carrying $\vectwo{-2}{1}$ forward and adding to $\vectwo{0}{2}$ yields
\[
v= \vectwo{-1}{1}+A\cdot\vectwo{-2}{3}+A^2\vectwo{-1}{2}
\]
Similarly, $\vectwo{-2}{3}=\vectwo{-1}{1}+A\vectwo{1}{2}$ gives
\[
v= \vectwo{-1}{1}+A\cdot\vectwo{-1}{1}+A^2\vectwo{0}{4}
\]
By replacing $\vectwo{0}{4}$ in $\Z^2$ with its expansion in $\bdd[A]$, one finds
\[
\vectwo{-156/25}{-8/25}= \vectwo{-1}{1}+A\cdot\vectwo{-1}{1}+A^2\vectwo{-1}{1}+A^3\vectwo{1}{0}+A^4\vectwo{2}{-1} \in \bdd[A].
\]
It should be noted that, for this process to work, one must know at least one arbitrary (not necessarily the shortest one) representation of the form $\sum_{j=0}^{k-1}A^j\bz_j$, $\bz_j \in \Z^2[A]$ of the vector $v \in \Z^2[A]$ to perform the subsequent digital expansion.

\subsection{Twisted sum of a rotational digit system with itself.}\label{subsec:example_twisted} Consider the matrix
\[
A=\begin{pmatrix}
    0 & -1\\
    1 & -1/2\\
\end{pmatrix},
\]
that is a companion matrix to its characteristic polynomial $\varphi_A(x)=x^2+x/2+1$. It is easy to see this is a generalized rotation: it is similar to rotation matrix via the transformation
\[
T=\begin{pmatrix}
    4 & 0\\
    1 & \sqrt{15}
\end{pmatrix}, \qquad T^{-1}AT = \begin{pmatrix*}[r]
-1/4 & -\sqrt{15}/4\\
\sqrt{15}/4 & -1/4\\.
\end{pmatrix*}
\]
The $\R^2$ norm $\norm{\bx}_{A^{-1}}=
\norm{T^{-1}\bx}_{euclidean}$ is then $A^{-1}$ invariant. Calculations similar to those in Section \ref{subsec:rot345} yield the residue group with respect of the lattice
\[
\ll  = \Z^2 \cap A\Z^2 = L\Z^2, \qquad L = \begin{pmatrix}
    2 & 2\\
    1 & 0\\
\end{pmatrix}
\]
is
\[
\rr = \Z^2/\ll = \left\{\vectwo{0}{0}, \vectwo{1}{0}\right\}.
\]
\begin{figure}[h]
	\includegraphics[scale=0.75]{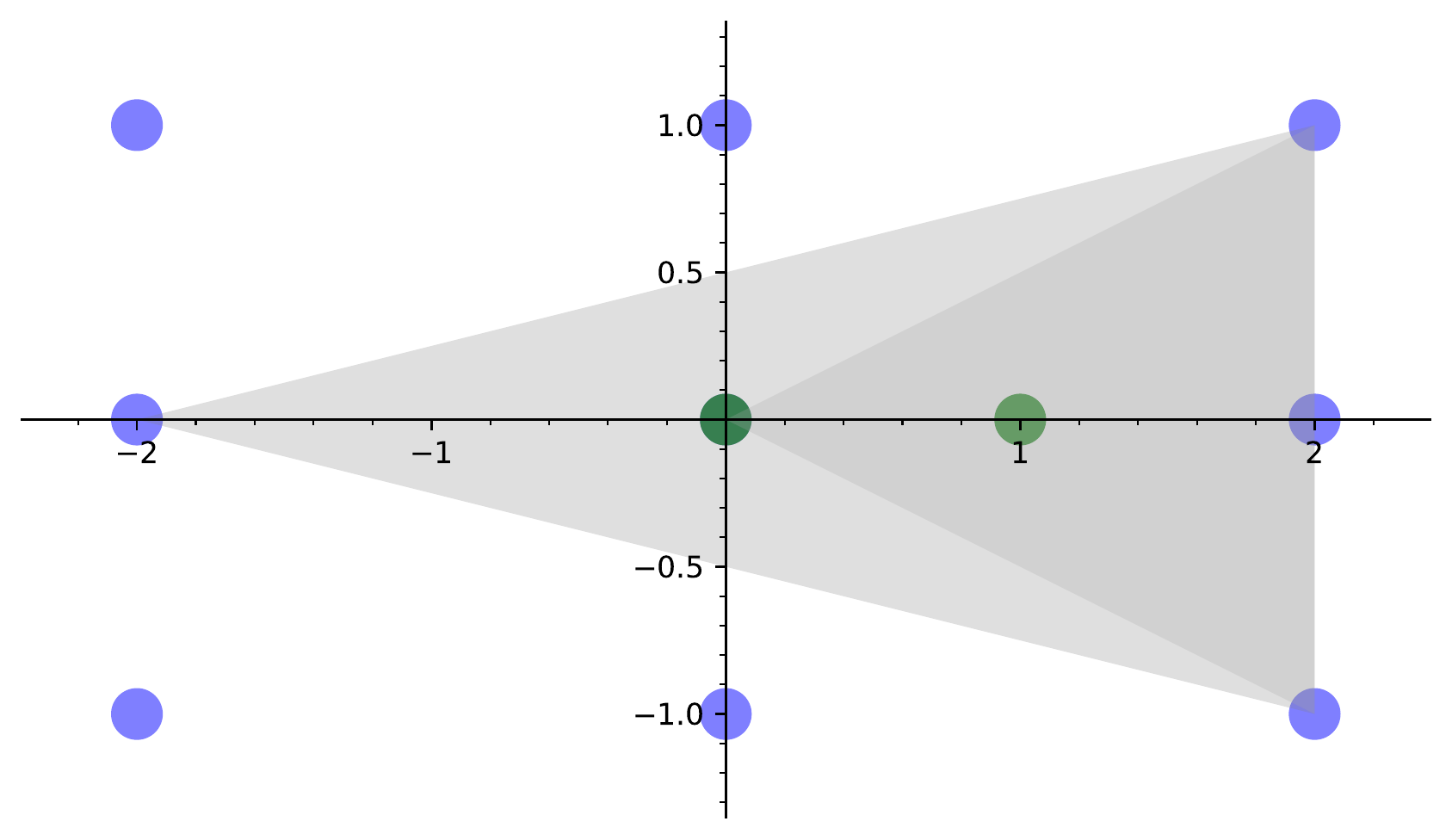} 
	\caption{Points from $\ll = \Z^2 \cap A\Z^2$ (blue) and from $\rr=\Z^2/\ll$ (green). Two small dark grey triangles, namely, $\tt_1$ and $\tt_2$  (light grey) enclose $(0,0)^T$ and $(0, 1)^T$, the elements of $\rr$.\label{fig_enclosures_2}}
\end{figure}
By enclosing these residue vectors inside the triangles with vertices
\[
\tt_1 := \left\{\vectwo{-2}{0}, \vectwo{2}{1}, \vectwo{2}{-1}\right\}, \qquad \tt_2: = \left\{\vectwo{0}{0}, \vectwo{2}{1}, \vectwo{2}{-1}\right\}, 
\]
as depicted in Figure \ref{fig_enclosures_2}, one finds the pre--periodic digit set
\[
\bdd\vectwo{0}{0}:=\vectwo{0}{0} - \tt_1, \qquad \bdd\vectwo{1}{0} := \vectwo{1}{0} - \tt_2,
\]
\[
\bdd' =  \bdd\vectwo{0}{0} \bigcup \bdd\vectwo{1}{0} = \left\{\vectwo{2}{0}, \vectwo{-2}{-1}, \vectwo{-2}{1}, \vectwo{1}{0}, \vectwo{-1}{-1}, \vectwo{-1}{1}\right\}.
\]
Then $\Phi_{\mathrm{r}, A}(\bx)=A^{-1}(\bx - \bd_{\mathrm{r},A}(x))$ is ultimately periodic in $\Z^2$. The subsequent analysis of the repeller points, similar to the one in Section \ref{subsec:rot345} carried out by solving corresponding linear programs yields $11$ integral repeller points
\[
\text{Rep}_{\Phi_{\mathrm{r}, A}} := \left\{ \vectwo{0}{-4}, \vectwo{0}{-3}, \vectwo{0}{-2},
\vectwo{0}{-1}, \vectwo{0}{0}, \vectwo{0}{1}, \vectwo{0}{2}, \vectwo{2}{4}, \vectwo{2}{5}, \vectwo{-1}{0}, \vectwo{1}{2}\right\},
\]
as illustrated in Figure \ref{fig_feasible_2}, (A) and (B).
\begin{figure}
	\begin{subfigure}{0.49\textwidth}
		\includegraphics[scale=0.75]{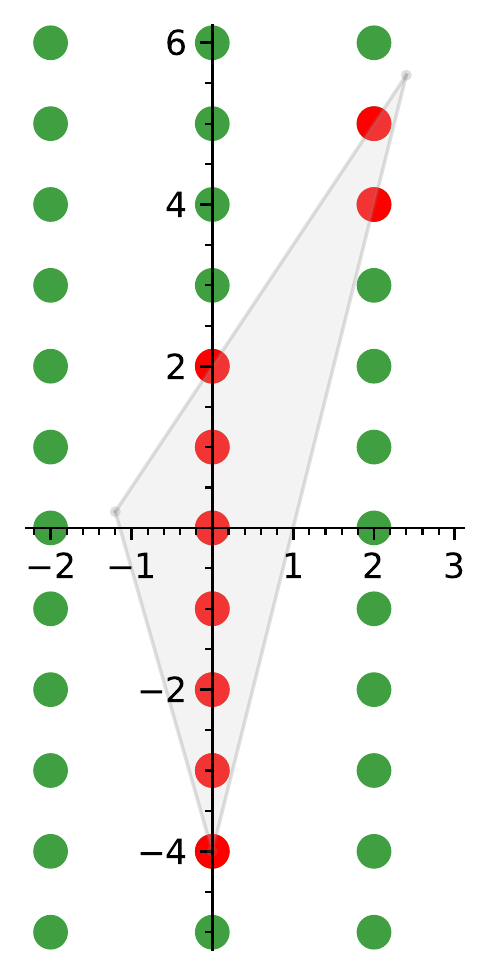} 
		\caption{$\bx=(0, 0)^T$}
	\end{subfigure}
	\hfill
	\begin{subfigure}{0.49\textwidth}
		\includegraphics[scale=0.75]{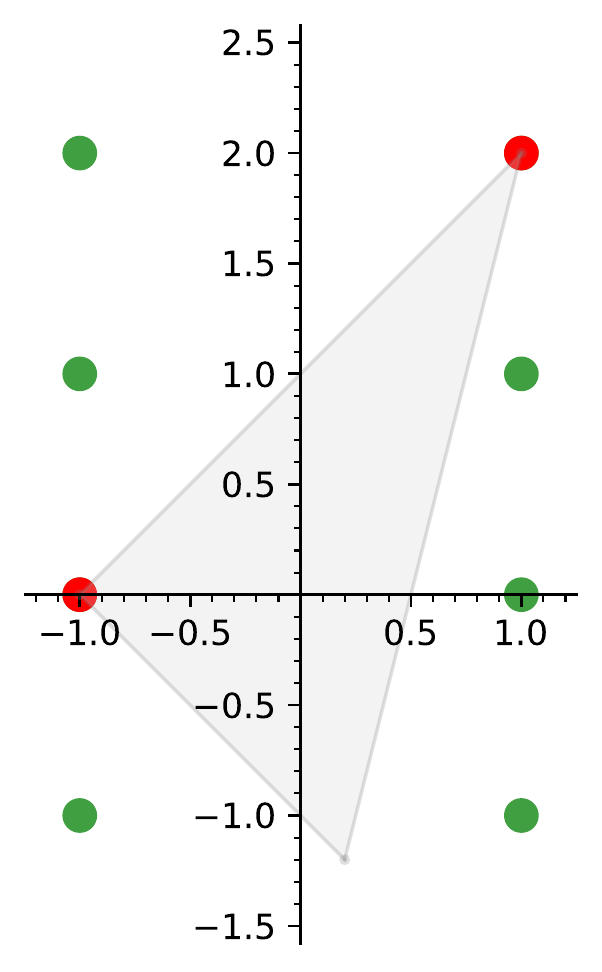} 
		\caption{$\bx=(1, 0)^T$}
	\end{subfigure}
	\caption{Repeller points $\rp{\Phi_{\mathrm{r},A}}{\bx}$  (in red) inside the regions of repelling action of $\Phi_{\mathrm{r}, A}$ (in grey), for each residue class $\bx \in \Z^2/\ll$. Remaining points $\bx +\ll$ are colored green.\label{fig_feasible_2}}
\end{figure}
It turns out, the orbits of each repeller point already contains at least one digit of $\bdd'$ that is depicted in \ref{fig_digits2}. Therefore, the digit system $(A, \bdd')$ with $6$ digits already has the finiteness property in $\Z^2[A]$.
\begin{figure}
	\includegraphics[scale=0.5]{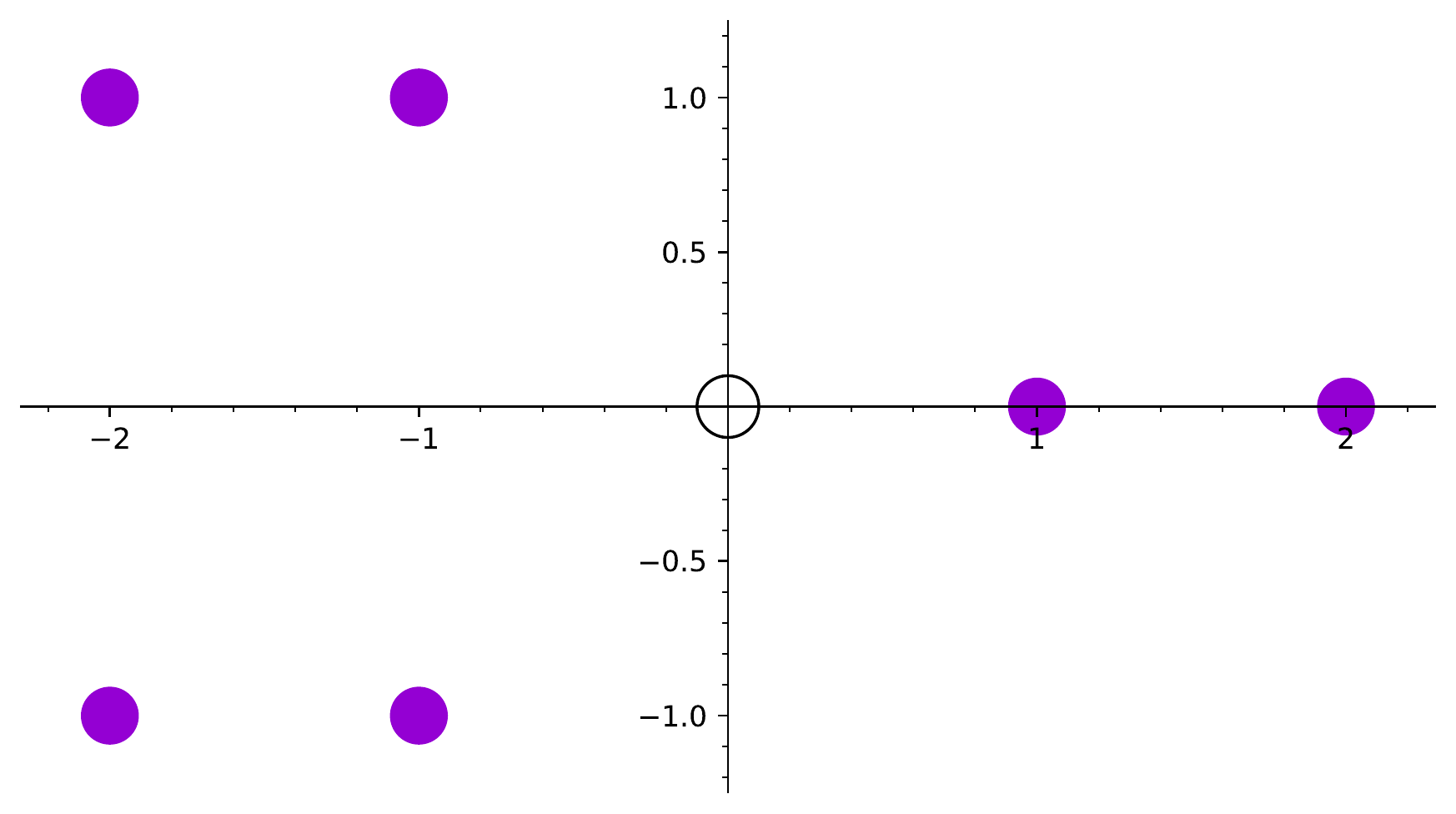} 
	\caption{The digit set $\bdd'$ for which $(A, \bdd')$ has finiteness property in $\Z^2[A]$. Notice again that $\bo_2 \not\in \bdd'$.\label{fig_digits2}}
\end{figure}

Next we append the zero vector $\bdd=\bdd' \cup \left\{\vectwo{0}{0}\right\}$. in order to construct the twisted sum of this digit system with itself. Consider the twisted sum $M=A \oplus_N A$, where
\[
N=\begin{pmatrix}
     0 & 1\\
     0 & 0\\
    \end{pmatrix}, \qquad M = A \oplus_N A =
    \begin{pmatrix}
        A & O_{2 \times 2}\\
        N & A\\
    \end{pmatrix}=
    \begin{pmatrix*}[r]
        0 & -1   & 0 & 0\\
        1 & -1/2 & 0 & 0\\
        0 &  1 & 0 & -1\\
        0 &  0 & 1 & -1/2
    \end{pmatrix*}.
\]
The matrix $M=H_2(\varphi_A)$ is an order--$2$ hypercompanion matrix to the polynomial $\varphi_A(x)$. Then the digit system $(M, \bdd \dsum \bdd)$ has the finiteness property in $\Z^4$ by Lemma \ref{lemTwist}: the radix expansions in $\Z^4$ are done using the twisted mapping $\Phi_{\mathrm{r}, M}=\Phi_{\mathrm{r}, A} \oplus_N \Phi_{\mathrm{r}, A}$ described in Section \ref{subsec:twistsum}. For instance,
\[
\vecfour{1}{2}{-3}{4}=\vecfour{1}{0}{-1}{1}+M\vecfour{2}{0}{2}{0}+M^2\vecfour{0}{0}{1}{0}+M^3\vecfour{0}{0}{-2}{-1}+M^4\vecfour{0}{0}{1}{0}
\]
The finiteness property of $(M, \bdd \oplus \bdd)$ extends to the whole module $\Z^4[M]$ through the mapping $\Psi_M$ \eqref{efExtDiv} by Corollary \ref{col:RestrExt}.

The digit set $\bdd \dsum \bdd$ is not the smallest possible digit set for which the digit system that has $M$ as its base matrix retains the finiteness property. It can be shown that it is possible to minimize the digit set size further in the following way. Use the digit set $\bdd_{1} = \bdd' \dsum \rr$ for vectors $\bx \oplus \by \in \Z^4$ as long as  $\bx \ne \bo_2$ and switch to the digit set $\bdd_{2} = \left\{\vectwo{0}{0}\right\} \dsum \bdd'$ as soon as  $\bx$
becomes $=\bo_2$ in the dynamical system $\Phi_{\mathrm{r}, M}$. Such a digit system $(M, \bdd'')$ with the digit set $\bdd'' = \bdd_{1} \cup \bdd_{r2}$ consists of $18$ digits and has the finiteness porperty in $\Z^4[M]$. It's highly likely $18$ is the minimal possible digit set size for which the finiteness property still holds in this particular base $M$.

The code for the examples provided in Section \ref{subsec:rot345} and \ref{subsec:example_twisted} is available to download online from \cite{Jank}.

\section{Concluding remarks}\label{sec:open}

We would like to conclude this paper by adding two open problems to the list of questions posed in \cite{JanThu} on  the arithmetics of module $\Z^d[A]$.

\begin{problem}\label{prob:p1}
Let $\bx \in \Q^d$ be such that every prime factor $p \in \N$ of the least common denominator of its coordinates divide the least common denominator of the entries in the matrix $A$. Is it true that such $\bx \in \Z^d[A]$?
\end{problem}

 Note that if there exists a prime number $p \in \N$, such that $p$ divides the least common denominator of the coordinates of $\bx$ and $p$ does not divide the denominator of any entry of $A \in \Q^{d \times d}$ then it is clear that $\bx \not  \in \Z^d[A]$.
 
 \begin{problem}\label{prob:p2}
 Let $A \in \Q^{d \times d}$ be non--degenerate. How to compute \emph{the stabilization index} $k_A$ in Lemma \ref{lem:Res}, namely, the smallest integer $k := k_A$, such that
 \[
\Z^d \cap A\Z^d[A] = \Z^d \cap \Z_k^d[A] = \Z^d \cap \left(\Z^d + A\Z^d+ \dots + A^{k-1}\Z^d \right)?
 \]
 \end{problem}
 We know how to find such $k_A$ for certain classes of matrices $A$, but the general case seems to be tied deeply to the representations of lattices (i.e. discrete subgroups) in locally compact groups.

Let us conclude with one final remark. If $(A, \bdd)$ is the classical standard digit system in $\Z^d$ for some expanding base matrix $A \in \Z^{d \times d}$ and some digit set $\bdd \subset \Z^d$, then the knowledge of the number of digits in the set $\bdd$ (which can be determined by observing the number of different symbols used in strings corresponding to radix expansions of random vectors) immediately yields the value of $\abs{\det{A}}$. Thus, it determines the possible conjugacy classes of $A$ in the group $\text{GL}(d, \Z)$. In contrast, for the rotational digit systems $(A, \bdd)$, $A \in \Q^{d \times d}$ with Property (F), described in Section \ref{secPerGenRot}, the size of the digit set depends on the attractor set $\aa_{\Phi_{\mathrm{r}}}$, which in turn depends (in a complicated way) on the choice of the coset representatives and their convex enclosures. Therefore, the size of $\bdd$ alone does not reveal so much information about the base matrix $A$ or its main characteristics, like the determinant or the dimension of $A$. It would be interesting to know if this could be useful in cryptography  applications (in particular, for scrambling the messages and encoding data streams, cf. \cite{Petho:91}).

\medskip 

{\bf Acknowledgement.}
We are grateful to the anonymous referee for her/his careful reading of the manuscript.


\bibliographystyle{ams}
\bibliography{comp1}

\end{document}